\newcommand{\showdate}{true}

\documentclass[reqno,a4paper]{amsart}
\usepackage{microtype}
\usepackage{color}
\usepackage{amssymb}
\usepackage{stmaryrd} 
\usepackage[utf8]{inputenc}
\usepackage{a4wide}
\usepackage{enumitem}
\usepackage[all]{xy}
\usepackage{ifthen}
\usepackage{graphicx}
\usepackage{slashed}
\RequirePackage{xspace}

\makeatletter
\DeclareRobustCommand\widecheck[1]{{\mathpalette\@widecheck{#1}}}
\def\@widecheck#1#2{%
    \setbox\z@\hbox{\m@th$#1#2$}%
    \setbox\tw@\hbox{\m@th$#1%
       \widehat{%
          \vrule\@width\z@\@height\ht\z@
          \vrule\@height\z@\@width\wd\z@}$}%
    \dp\tw@-\ht\z@
    \@tempdima\ht\z@ \advance\@tempdima2\ht\tw@ \divide\@tempdima\thr@@
    \setbox\tw@\hbox{%
       \raise\@tempdima\hbox{\scalebox{0.9}[-0.9]{\lower\@tempdima\box
\tw@}}}%
    {\ooalign{\box\tw@ \cr \box\z@}}}
\makeatother

\newcommand{\calz}{\mathcal{Z}}
\newcommand{\cala}{\mathcal{A}}
\newcommand{\calb}{\mathcal{B}}
\newcommand{\calm}{\mathcal{M}}
\newcommand{\calo}{\mathcal{O}}
\newcommand{\caln}{\mathcal{N}}
\newcommand{\cale}{\mathcal{E}}
\newcommand{\calg}{\mathcal{G}}
\newcommand{\calw}{\mathcal{W}}
\newcommand{\calk}{\mathcal{K}}
\newcommand{\calt}{\mathcal{T}}
\newcommand{\calc}{\mathcal{C}}
\providecommand{\texorpdfstring}[2]{#1}

\newcommand{\fkm}{$(4n{-}1)$-manifold}
\newcommand{\tkc}{$(n{-}1)$-connected}
\newcommand{\bmp}{Bianchi-Massey tensor}
\newcommand{\trc}[1]{{[#1]}}
\newcommand{\wh}{\widehat}

\DeclareMathOperator{\ann}{Ann}
\newcommand{\ie}{\emph{i.e.} }

\newcommand{\eg}{\emph{e.g.} }
\newcommand{\cf}{\emph{cf.} }
\newcommand{\mq}{q}
\newcommand{\princ}{\mathcal{F}}
\newcommand{\iprinc}{{\hspace*{0pt} \mkern3mu{\overline {\mkern-3mu \princ \mkern1mu}}\mkern-1mu }}
\newcommand{\pol}{P}
\newcommand{\ext}{\Lambda}
\newcommand{\sym}{\textup{Sym}}
\newcommand{\alt}{\textup{Alt}}
\newcommand{\gsym}{\calg}
\newcommand{\agsym}{\slashed\calg}
\newcommand{\grad}{\textup{Grad}}
\newcommand{\agrad}{\textup{Anti}}
\newcommand{\symp}{\mkern1.5mu{\cdot}\mkern1.5mu}
\newcommand{\ksmap}{\psi}
\newcommand{\poltp}{}
\newcommand{\esymp}{}
\newcommand{\kersym}[1]{K[#1]}
\newcommand{\kerasym}[1]{\slashed K[#1]}
\newcommand{\kercyc}[1]{C[#1]}
\newcommand{\bmsp}{\calb}
\newcommand{\bmspad}[2]{\calb^{#1}(#2)}
\newcommand{\bmspa}[1]{\bmspad{*}{#1}}
\newcommand{\bmspam}[1]{\bmspad{m{+}1}{#1}}
\newcommand{\bmspaf}[1]{\bmspad{4n}{#1}}
\newcommand{\bmspakd}[2]{\calb_k^{#1}(#2)}
\newcommand{\bmspak}[1]{\bmspakd{*}{#1}}
\newcommand{\bmspakm}[1]{\bmspakd{m+1}{#1}}
\newcommand{\tripsp}{\kersym{\kc^* \otimes H^*}}
\newcommand{\trip}{\calt}

\newcommand{\dga}{\cala}
\newcommand{\tdga}{\caln}
\newcommand{\edga}{\cale}
\newcommand{\drpl}{\Omega_{\textup{PL}}}
\newcommand{\mnalg}{\calm}
\newcommand{\clalg}{\calz}
\newcommand{\ccalg}{\calc}
\newcommand{\rp}{\alpha}
\newcommand{\prp}{\gamma}
\newcommand{\csq}{c}
\newcommand{\qi}{q}

\newcommand{\kc}{E}
\newcommand{\ke}{e}
\newcommand{\ikc}{{\hspace*{0pt} \mkern4mu{\overline {\mkern-4mu E}}}}
\newcommand{\icsq}{{\hspace*{0pt} \mkern2mu{\overline {\mkern-2mu c \mkern1mu}} \mkern-1mu}}

\newcommand{\pmt}{\widetilde{p}_1}

\newcommand{\xra}{\xrightarrow}

\newcommand{\ogr}{\widetilde{Gr}}
\newcommand{\PP}{\mathbb{P}}
\newcommand{\CP}{\C P}

\newcommand{\an}[1]{\langle #1 \rangle}

\newcommand{\half}{{\textstyle\frac{1}{2}}}

\newcommand{\third}{{\textstyle\frac{1}{3}}}
\newcommand{\twth}{{\textstyle\frac{2}{3}}}

\newcommand{\Z}{\mathbb{Z}}
\newcommand{\Q}{\mathbb{Q}}
\newcommand{\bbq}{\mathbb{Q}}

\newcommand{\R}{\mathbb{R}}
\newcommand{\C}{\mathbb{C}}

\newcommand{\into}{\hookrightarrow}

\newcommand{\Id}{\textup{Id}}
\newcommand{\del}{\partial}

\newcommand{\wt}[1]{\widetilde #1}

\newcommand{\ol}{\overline}

\newcommand{\gen}[1]{\langle#1\rangle}

\newtheorem{thm}{Theorem}[section]
\newtheorem{prop}[thm]{Proposition}
\newtheorem{lem}[thm]{Lemma}
\newtheorem{conj}[thm]{Conjecture}
\newtheorem{cor}[thm]{Corollary}

\theoremstyle{definition}
\newtheorem{defn}[thm]{Definition}

\theoremstyle{remark}
\newtheorem{rmk}[thm]{Remark}
\newtheorem*{rmk*}{Remark}
\newtheorem{ex}[thm]{Example}

\setlist{leftmargin=*}
\title[The rational homotopy type of $(n{-}1)$-connected manifolds]{The rational homotopy type of $(n{-}1)$-connected manifolds of dimension up to $5n{-}3$}
\author[D. Crowley]{Diarmuid Crowley}
\address{School of Mathematics and Statistics,
University of Melbourne,
Parkville, VIC, 3010, Australia}
\email{dcrowley@unimelb.edu.au}
\author[J. Nordström]{Johannes Nordström}
\address{Department of Mathematical Sciences,
University of Bath,
Bath BA2 7AY, UK}
\email{j.nordstrom@bath.ac.uk}

\begin{document}

\begin{abstract}
We define the \bmp{} of a topological space $X$ to be a linear
map $\calb \to H^*(X)$, where $\calb$ is a subquotient of $H^*(X)^{\otimes 4}$
determined by the algebra $H^*(X)$. 
We then prove that if $M$ is a closed 
\tkc{} manifold of dimension at most $5n{-}3$ (and $n \geq 2$) then its
rational homotopy type is determined by its cohomology algebra and \bmp,
and that $M$ is formal if and only if the \bmp{} vanishes.

We use the Bianchi-Massey tensor to show that there are many \tkc{} \fkm s 
that are not formal but have no non-zero Massey products, and to present a
classification of simply-connected $7$-manifolds up to finite ambiguity.
\end{abstract}

\ifthenelse{\boolean{\showdate}}{\date{\today}}{}
\maketitle

\ifthenelse{\boolean{\showdate}}{\vspace{-0.8\baselineskip}}{}
\vspace{-\baselineskip}

\section{Introduction}

This paper is concerned with the rational homotopy theory of closed 
\tkc{} manifolds of dimension up to $5n{-}3$ for $n \geq 2$.
A continuous map $f : X \to Y$ is a rational homotopy equivalence
if the induced maps $f_* : \pi_k(X) \otimes \Q \to \pi_k(Y) \otimes \Q$ are
isomorphisms. If the spaces are simply-connected then this condition is
equivalent to $f^* : H^*(Y) \to H^*(X)$ being an isomorphism of the cohomology
algebras (throughout the paper, we use cohomology with rational coefficients
unless explicitly stated otherwise).
A further fundamental rational homotopy invariant is the
Massey product structure on $H^*(X)$. In particular, 
non-trivial Massey products are an obstruction to $X$ being \emph{formal} in
the sense of Sullivan \cite{sullivan77} (see \S \ref{subsec:models}).

Miller \cite{miller79} proved that, for $n \geq 2$, any closed \tkc{} manifold
of dimension $\leq 4n{-}2$ is formal. On the other hand,
it was well known that there are examples of non-formal
closed \tkc{} manifolds of dimension $\geq 4n{-}1$
\cite{kreck91a, dranishnikov05}.
Closed \tkc{} \fkm s therefore represents a borderline situation, the
simplest non-trivial case from the point of view of rational homotopy.
In this paper we are concerned with closed \tkc{} manifolds whose dimensions
range from $4n{-}1$ to $5n{-}3$.
In this range the only possible non-trivial Massey products are triple
products, but these do not in general suffice for determining the rational
homotopy type.

The starting point of this paper is the definition
of what we term the \emph{\bmp}.
It is similar in style to the definition of Massey triple products,
but unlike Massey products it is completely independent of auxiliary choices.
The \bmp{} captures precisely the information needed to determine the rational
homotopy type of a closed \tkc{} manifold $M$ of dimension $m \leq 5n{-}3$, and
in particular it is a \emph{complete} obstruction to formality of such
manifolds.
Moreover, the \bmp{} can be computed directly from
the cohomology ring of a coboundary $W$ for $M$ such that
the restriction map $H^*(W) \to H^*(M)$ is surjective in degree
$\leq m{+}1{-}3n$.
This makes the determination of the rational homotopy type tractable
for many examples.

Throughout the paper, all manifolds will be assumed to be connected.
All graded algebras and rings will also be assumed to be connected, in the
sense that the degree 0 part has rank 1, and moreover graded commutative.

\subsection{The \bmp}

We will define the \bmp{} on the cohomology of a differential
graded algebra.
Let us first summarise some notation for various spaces of tensors,
set up in further detail in~\S \ref{subsec:quad}.
For a graded vector space $V$
let $\calg^k V$ denote the quotient of $V^{\otimes k}$ by relations of
graded symmetry, \eg if $V$ is concentrated in odd degree then $\calg^k V$ is
the $k$th exterior power $\ext^k V$, while if $V$ is concentrated in even
degree then $\calg^k V = \pol^k V$, the space of homogeneous degree $k$
polynomials. $\calg^k V$ is itself a graded vector space.
There is a linear map $\calg^2 \calg^2 V \to \calg^4 V$ given by 
full graded symmetrisation, and
we denote its kernel by $\kersym{\calg^2\calg^2 V}$.

\begin{rmk*}
If $V$ is concentrated in odd degree, then
$\kersym{\calg^2\calg^2 V} = \kersym{\pol^2\ext^2V}$ can be
identified with the subspace of $V^{\otimes 4}$ consisting
of tensors that satisfy the symmetries of the Riemann curvature tensor, in
particular the first Bianchi identity. 
\end{rmk*}

Given a (graded commutative) graded algebra $H^*$, the product
$H^* \otimes H^* \to H^*$ by definition 
factors through a map
$\calg^2 H^* \to H^*$. We denote its kernel by
$\kc^* \subseteq \calg^2 H^*$, and let
\begin{equation}
\label{eq:bmsp}
\bmspa{H^*} := \calg^2 \kc^* \cap \kersym{\calg^2\calg^2 H^*} .
\end{equation}
When $H^*$ is the cohomology algebra of a topological space or (graded commutative) differential graded algebra $\bullet$, we will use
$\bmspa{\bullet}$ as a short-hand for $\bmspa{H^*(\bullet)}$.

Given a differential graded algebra $(\dga, d)$ over $\Q$,
let $\clalg^k := \ker d \subseteq \dga^k$, the space of closed elements of
degree $k$.
Pick a right inverse $\rp \colon H^*(\dga) \to \clalg^*$ for the projection to
cohomology. This induces a map $\rp^2 \colon \calg^2 H^*(\dga) \to \clalg^*$,
taking exact values precisely on $\kc^*$; there is a degree $-1$ linear map
$\prp : \kc^* \to \dga^{*-1}$ such that $\rp^2(\ke) = d\prp(\ke)$ for
$\ke \in \kc^*$.
Now observe that the map $\calg^2 E^* \to \dga^{*-1}$ induced by
the graded symmetrisation of
\[ \kc^* \otimes \kc^* \to \dga^{*-1},
\; \ke \otimes \ke' \mapsto \prp(\ke)\rp^2(\ke') \]
takes closed values on $\bmspa{\dga} \subseteq \calg^2\kc^*$.
It is easy to see that the induced degree $-1$ map
\begin{equation}
\label{eq:bmp}
\princ : \bmspa{\dga} \to H^{*-1}(\dga)
\end{equation}
is independent of the choice
of $\prp$. It is not as obvious, but nevertheless true, that $\princ$ is also
independent of the choice of $\rp$ (Lemma \ref{lem:principal}).

\begin{defn}
\label{def:bmp}
The \emph{\bmp} of the DGA $(\dga, d)$
is the linear map defined in \eqref{eq:bmp}.
\end{defn}

Any algebra homomorphism $H^* \to H'^*$ maps $\kc^* \to \kc'^*$ and
hence $\bmspa{H^*} \to \bmspa{H'^*}$. In particular,
if $\phi : \dga \to \dga'$ is a DGA homomorphism then the induced map
$\phi_\#$ on cohomology maps 
$\bmspa{\dga}$ to $\bmspa{\dga'}$, and the \bmp{} is clearly
functorial in the sense that the diagram below commutes:
\vspace{-2mm}
\[ \xymatrix{
\bmspa{\dga} \ar[d]^{\princ} \ar[r]^{\phi_\#}
& \bmspa{\dga'} \ar[d]^{\princ'} \\
H^{*-1}(\dga) \ar[r]^{\phi_\#} & H^{*-1}(\dga')}  \]
The definition of formality therefore immediately implies that the \bmp~of
$\dga$ must be trivial if $\dga$ is formal.

\subsection{Determining the rational homotopy type}

Any simply-connected topological space $X$ has a
rationalisation $(X_\Q, f)$ 
(unique up to homotopy, see \eg \cite[Theorem 9.7]{felix01}), 
which is a simply-connected space $X_\Q$ together with a map $f : X \to X_\Q$ such that
$f_* : \pi_k(X) \otimes \Q \to \pi_k(X_\Q)$ are isomorphisms.
Two spaces $X$ and $X'$ are rationally homotopy equivalent if and only if their
rationalisations are homotopy equivalent.

For any CW complex $X$, we can define the \bmp{}
$\princ_X \colon \bmspa{X} \to H^{*-1}(X)$ in
terms of the algebra $\drpl(X)$ of piecewise linear forms on $X$ (see Sullivan
\cite[\S 7]{sullivan77} or F\'elix--Halperin--Thomas \cite[II 10(c)]{felix01}
for the definition of $\drpl(X)$).
The theorem below identifies the \bmp{} as a complete obstruction to
realising an isomorphism of the cohomology algebras of closed \tkc{} manifolds
of dimension up to $5n{-}3$ by a rational homotopy equivalence. Such obstructions
are studied more generally by Halperin and Stasheff \cite{halperin79}.

Note that if $M$ is \tkc{}, then $\bmspa{M}$ is trivial in degrees
$\leq 4n{-}1$. If $m \leq 5n{-}3$, then by Poincar\'e duality the only non-trivial
part of $H^*(M)$ in degree $\geq 4n{-}2$ is $H^m(M)$, so the only non-trivial
component of $\princ$ is $\bmspam{M} \to H^m(M)$. (However, even without any
connectedness hypothesis it is still the case that the $\bmspam{M} \to H^m(M)$
component of $\princ$ determines $\princ$ completely using Poincar\'e duality,
see Lemma \ref{lem:unitrip_to_bm}.) 

\begin{thm}
\label{thm:type}
For $n \geq 2$, the rational homotopy type of a closed simply-connected,
rationally \tkc{} manifold $M$ of dimension $m \leq 5n{-}3$ is
determined by its cohomology algebra $H^*(M)$ and the top component of
the \bmp{}, $\princ_M \colon \bmspam{M} \to H^m(M)$.
More precisely, if $M$ and $M'$ are closed \tkc{} $m$-manifolds and
$G \colon H^*(M) \to H^*(M')$ is an isomorphism then there exists a
homotopy equivalence $g \colon M'_\Q \to M_\Q$ of the rationalisations
such that $G = g^*$ if and only if the diagram below commutes.
\[ \xymatrix{
\bmspam{M} \ar[d]^{\princ_M} \ar[r]^{G}
& \bmspam{M'} \ar[d]^{\princ_{M\!'}} \\
H^m(M) \ar[r]^{G} & H^m(M')}  \]
\end{thm}

We deduce Theorem \ref{thm:type} from Corollary \ref{cor:minmodel}, which
characterises the minimal model of $M$ in terms of $H^*(M)$ and the \bmp.
In Corollary \ref{cor:obstruction}, that picture of the minimal model also lets
us understand when $M$ is formal.

\begin{thm}
\label{thm:obstruction}
For $n \geq 2$, a closed \tkc{} $M$ of dimension $m \leq 5n{-}3$ is formal if
and only if its \bmp{} $\princ_M \colon \bmspam{M} \to H^m(M)$ is trivial.
\end{thm}

In the special case of simply-connected 7-manifolds whose 
product ${H^2(M) \times H^2(M) \to H^4(M)}$ is trivial, this was already
argued by Muñoz and Tralle \cite[Theorem 14]{munoz14}.
More generally, on the other hand, Kadeishvili \cite{kadeishvili80} proved that
one can define an $A_\infty$-algebra structure on the cohomology of any
topological space, whose equivalence class determines the rational homotopy
type of the space.
In particular, by \cite[Proposition 7]{vallette12} 
the space is formal if the auxiliary choices in
the definition can be made so that all the higher-order products vanish 
(a precise interpretation of the slogan that ``a space is formal if and only
if the Massey products vanish uniformly''.)

One perspective on the \bmp{} is that it identifies the components of
the \mbox{$A_\infty$-structure} that are significant in the context of \tkc{}
manifolds of dimension $\leq 5n{-}3$, see \S \ref{subsec:ainfty}.
Discarding the components that depend on choices is useful for understanding
examples, and in the applications discussed below.

\begin{rmk}
For manifolds that are not simply-connected, the \bmp{} is still
a well-defined invariant, but the relation between the rational homotopy type
and the minimal model is less straight-forward. When we say we consider \tkc{}
manifolds, we will therefore always require $n \geq 2$.
\end{rmk}

\subsection{Realisation}
We next turn to the question of realisation of invariants of the above type.
In \S \ref{subsec:realise} we apply a modification of Sullivan's methods
for the realisation of rational models by closed simply-connected manifolds
\cite[Theorem 13.2]{sullivan77}
to obtain 
Theorem~\ref{thm:realisation} below on realisation by \tkc{} manifolds.
The proof of Theorem \ref{thm:realisation} relies on work Su~\cite{su09}.

Note that when the dimension is divisible by 4 there are additional
\emph{a priori} necessary conditions. If $H^*$ is a $4k$-dimensional Poincar\'e
duality algebra then $H^{2k}$ inherits the structure of a non-degenerate
quadratic form over $\Q$ once we choose a generator $\alpha \in (H^{4k})^\vee$.
To realise the algebra as the rational cohomology of a closed oriented
manifold, it is clearly necessary that we can choose $\alpha$ so that $H^{2k}$
is isometric over $\Q$ to a non-singular integral form (equivalently to a
diagonal form $\Sigma \pm \! x_i^2$).
If we also want to prescribe rational Pontrjagin classes,
then we must be able to choose $\alpha$ so that in addition the associated
Pontrjagin numbers are integers which satisfy Hirzebruch's signature theorem
and which are the Pontrjagin numbers of some closed \tkc{} manifold.
 
\begin{thm}
\label{thm:realisation}
Let $n \geq 2$ and $m \leq 5n-2$.
Let $H^*$ be an $m$-dimensional rational Poincar\'e duality algebra that is
\tkc, \ie  $H^0 = \Q$ and $H^k = 0$ for $0 <\! k \!< n$.
Let $p_i \in H^{4i}$ for $4i \leq m$,
and let $F \colon  \bmspam{H^*} \to H^m$ be a linear map.
If $m = 4k$ assume in addition that there is an $\alpha \in (H^{4k})^\vee$ such
that the following hold:
\begin{enumerate}
\item the induced quadratic form on $H^{2k}$ is isometric over $\Q$ to
a sum of squares $\Sigma \pm\!x_i^2$;
\item $(H^*, p_*, \alpha)$ satisfies the signature theorem;
\item 
the numbers $\alpha(p_{i_1} \dots p_{i_r})$ for $\sum_j i_j = k$ are integers 
realised as the Pontrjagin numbers of some closed smooth \tkc{} $4k$-manifold.
\end{enumerate}
Then there is a closed smooth \tkc{} $m$-dimensional manifold $M$ with
rational Pontrjagin classes $p_*(M)$ and \bmp{} $\princ_M$ such that
\[ (H^*(M), p_*(M), \princ_M) \cong (H^*, p_*, F).\]
\end{thm}

Note that the above realisation statement applies in a wider dimension range
than the preceding classification statements. However, for an \tkc{}
Poincar\'e algebra $H^*$ of dimension $m \geq 5n{-}1$ there can be obstructions
to realising a map $\bmspam{H^*} \to H^m$ as the \bmp{} of a rational Poincar\'e space.

\begin{ex}
\label{ex:obstr}
Let $X$ be a connected rational Poincar\'e space of dimension 
$5n-1$ with $H^n(X) \cong \Q^3$
and $H^{2n}(X) \cong \Q^2$, with bases
$\{x_1, x_2, x_3\} \subset H^n(X)$ and $\{y_1, y_2\} \subset H^{2n}(X)$ 
such that $x_1x_3 = y_1$, $x_2x_3 = y_2$ and
$x_1^2 = x_2^2 = x_3^2 = x_1x_2 = 0$. 
Then $(x_1y_1)x_2^2 - (y_1x_2)(x_2x_1) - (x_2y_2)x_1^2 + (x_1y_2)(x_1x_2)$ 
is a non-zero element of $\bmspad{5n}{X}$, and it is explained in
Example \ref{ex:obstr_detail} that $\princ$ must vanish on this element.
\end{ex}

Thus as we increase the dimension of $M$, the rational homotopy type is harder
to describe not only because we need to add data to capture the higher-order
Massey products, but also because the constraints on realising $\princ$
become more opaque. We do not explore these issues further in this paper.

We can also consider the problem of \emph{integral} realisation of \bmp s.
For a torsion-free ring $F^*$, we can define $\bmspa{F^*}$ entirely analogously
to \eqref{eq:bmsp}. When $F^*$ is the free part of $H^*(M;\Z)$, we abbreviate
this to $\bmspa{M;\Z}$.
We henceforth implicitly assume that all manifolds are oriented and
define the ``integral restriction'' $\iprinc_M : \bmspam{M;\Z} \to \Q$ as the
composition $\bmspam{M;\Z} \to \bmspam{M} \to H^m(M) \to \Q$,
where the second map is $\princ_M$ and the third is
integration over the fundamental class. 

\begin{rmk}
Treating the \bmp{} of a closed oriented $m$-manifold $M$ as an element of
$\bmspa{M}^\vee$ is not natural in the context of
Theorem \ref{thm:type}---since the fundamental class is not invariant under
rational homotopy equivalence---but it is in the context of diffeomorphism
classification and/or cohomology with integer coefficients, where $\princ_M$
and $\iprinc_M$ are equivalent.
\end{rmk}

Even under our simplifying connectivity assumptions,
the question of the realisation of integral Poincar\'e duality
rings by simply-connected manifolds is a hard problem.
The questions that motivated our work really concern the lowest-dimensional
non-trivial case, \ie simply-connected 7-manifolds. However, the arguments
say something at least in the more general critical case of \tkc{} \fkm s,
so we state our integral realisation claim in that context.

We focus on the minimal integral data required to support the \bmp{}
of an \tkc{} \fkm{} $M$.
Let $TH^*(M; \Z) \subset H^*(M; \Z)$ be the torsion subgroup,
$FH^*(M; \Z) := H^*(M; \Z)/TH^*(M; \Z)$,
and $FH^{n*} \subset FH^*$ the subring of elements in degree divisible by $n$.
For an \tkc{} \fkm{} we have $FH^{n*} = H^0 \oplus H^n \oplus FH^{2n}$,
and $\bmspaf{FH^{n*}} = \bmspaf{M; \Z}$.

\begin{thm}
\label{thm:realiseM}
Let $F^{n*}$ be a connected torsion-free graded ring concentrated in degrees
$0$, $n$ and~$2n$, and $\iprinc : \bmspaf{F^{n*}} \to \Q$ a homomorphism.
Then there exists some closed \tkc{} 
$M^{4n-1}$ with an isomorphism $FH^{n*}(M; \Z) \cong F^{n*}$
that identifies $\iprinc_M$ with $\iprinc$.

Moreover,
$H^n(M; \Z)$ can be taken to be torsion-free, \ie $H^n(M; \Z) \cong F^n$.
In addition, if $\iprinc_M$ is integer-valued then 
$H^{2n}(M; \Z)$ can be taken to be torsion-free too.
\end{thm}

In \S \ref{sec:coboundaries} we first describe how the \bmp{} of the boundary
of a suitable compact manifold $W$ can be computed in terms of the cohomology
ring of $W$, and then construct
the required $M$ as a boundary of a $4n$-manifold.
See Theorem \ref{thm:realiseMM} for more details on the cohomology ring
of the \fkm{} produced by this argument.


\subsection{Classification up to finite ambiguity}
\label{subsec:classification_up_to_finite}
One of the motivations of Sullivan's work on rational homotopy theory is the
principle that the rational homotopy type of a simply-connected manifold together with some
characteristic class and integral data determines the diffeomorphism type up to
finite ambiguity, \eg \cite[Theorem 13.1]{sullivan77} classifies smooth
manifolds up to finite ambiguity in terms of their rational homotopy type,
rational Pontrjagin classes, bounds on torsion and certain integral lattice
invariants. 

Kreck and Triantafillou \cite{kreck91a} work with less than the 
full rational homotopy type and present stronger results,
\eg \cite[Theorem 2.2]{kreck91a}, where less
of the lattice data is required explicitly, or can be replaced by parts of
the integral cohomology ring.
In the first instance, we are not too concerned about how little of the 
integral cohomology ring $H^*(M;\Z)$ one needs to remember; in Proposition
\ref{prop:classification_up_to_finite}
we explain how to deduce the following result.

\begin{thm}
\label{thm:classification_up_to_finite}
For $n \geq 2$ and $m \leq 5n{-}3$, 
the diffeomorphism type of a closed \tkc{} $m$-manifold $M$ 
is determined up to finite ambiguity by its integral cohomology ring
$H^*(M;\Z)$, Pontrjagin classes ${p_k(M) \in H^{4k}(M;\Z)}$ and 
\bmp{} $\iprinc_M : \bmspam{M; \Z} \to \Q$; \ie given
such an $M$, the set of \tkc{} $m$-manifolds $M'$ with a ring isomorphism
$G : H^*(M';\Z) \to H^*(M;\Z)$ such that $G(p_k(M')) = p_k(M)$, and
$G^\# \iprinc_M = \iprinc_{M'}$ contains only finitely many diffeomorphism
classes.
\end{thm}

For simply-connected $7$-manifolds, we can simplify the invariants from
Theorem \ref{thm:classification_up_to_finite} to align them more closely
with the realisation statement Theorem \ref{thm:realiseM}.
In this case Poincar\'{e} duality means that $H^*(M)$ is determined up to
isomorphism by the rational cup square $\pol^2 H^2(M) \to H^4(M)$.
Hence given a bound on the size of $TH^*(M; \Z)$, the full cohomology ring
$H^*(M; \Z)$ is determined up to finite ambiguity by the truncated ring
$FH^{2*}(M;\Z) = H^0(M;\Z) \oplus H^2(M; \Z) \oplus FH^4(M; \Z)$. 
Similarly $p_1(M) \in H^4(M; \Z)$ is determined up to finite ambiguity by its
image $\pmt(M) \in FH^4(M; \Z)$, and
by Theorem \ref{thm:classification_up_to_finite} we have

\begin{cor}
\label{cor:1c7m_up_to_finite}
For all $N \geq 0$,
closed $1$-connected $7$-manifolds $M$ with $|TH^*(M; \Z)| < N$
are classified up to finite ambiguity by $FH^{2*}(M;\Z)$,
$\pmt(M) \in FH^4(M; \Z)$ and
the \bmp{} $\iprinc_M \colon \bmspad{8}{M;\Z} \to \Q$.
\end{cor}

In \S \ref{subsec:7mfds} we build on Theorem \ref{thm:realiseM} to also
study which $\pmt$ are realised. 
Proposition \ref{prop:1c7mFA} gives a satisfactory understanding of which
integral invariants are realised by simply-connected spin 7-manifolds, and
we then discuss directions for a complete classification of
such manifolds.

%

\subsection{Non-formal manifolds with only trivial Massey products}

We will elaborate on the relationship between the \bmp{} 
and the Massey triple products of elements of $H^*(M)$ in
\S \ref{subsec:ordinary}. To simplify the notation, let us limit the present
summary to the critical case of \tkc{} \fkm s, where the triple products only
involve classes in $H^n(M)$.
The triple product $\gen{x,z,y}$ of $x, y, z \in H^n(M)$ is defined if and only
if $xz = yz = 0 \in H^{2n}(M)$. 
Considered as an element of $H^{3n-1}(M)$, the triple product in general
depends on auxiliary choices of cocycle representatives, but
$\gen{x, z, y}$ is well-defined considered as an element of the quotient space
$H^{3n-1}(M)/(xH^{2n-1}(M) + y H^{2n-1}(M))$.
Therefore if $w \in H^n(M)$ has $xw = yw= 0$ then we get a well-defined
\begin{equation}
\label{eq:ord_tensor}
\gen{x,z,y}w \in H^{4n-1}(M) \cong \bbq .
\end{equation}
If the cup product $\csq : \gsym^2 H^n(M) \to H^{2n}(M)$ is trivial then
\eqref{eq:ord_tensor} defines an element $\mq \in (H^n(M)^\vee)^{\otimes 4}$.
It is graded anti-symmetric under swapping $x \leftrightarrow y$ or
$z \leftrightarrow w$, symmetric under swapping both $x \leftrightarrow z$
and $y \leftrightarrow w$, and also satisfies the Bianchi identity.
In \S \ref{subsec:quad}, we explain that the space of such tensors
($\kerasym{\sym^2\agrad^2H^n(M)^\vee}$ in the notation there) is naturally dual
to $\kersym{\pol^2\gsym^2H^n(M)}$.
When $\csq$ is trivial, the latter space equals $\bmspaf{M}$, and
$\princ$ and $\mq$ correspond under this duality.
If $\csq$ is not trivial, then the duality can still be used to determine all
defined values of \eqref{eq:ord_tensor} from~$\princ$. In turn, we can use
Poincar\'e duality to determine all Massey triple products
from~\eqref{eq:ord_tensor}.

A basic point of this paper is that the $\princ$ side of this duality is more
useful when $\csq$ is non-trivial: while the \bmp{} still
determines the Massey triple products, the converse is not true in general.
In \S \ref{subsec:nonordinary} we study how the domain $\bmspaf{M}$
of the \bmp{} of an \tkc{} \fkm, and the set of defined Massey triple products,
depend on the kernel $\kc^{2n} \subseteq \gsym^2 H^n(M)$ of~$\csq$. 
Combined with the above realisation results, one finds
that there are many examples of closed \fkm s with non-trivial \bmp{}
but no non-trivial Massey products. The examples that are
simplest to describe have $n = 2k + 1$ odd.

\begin{ex}
Combining Example \ref{ex:rk4} and Theorem \ref{thm:realiseM}, there exists
for each $k \geq 1$ a non-formal $2k$-connected $(8k{+}3)$-manifold
$M$ with 
$H^{2k+1}(M;\Z) \cong \Z^4$ and $H^{4k+2}(M;\Z) \cong \Z^3$
such that $\csq : \ext^2 H^{2k+1}(M;\Z) \to H^{4k+2}(M;\Z)$ is equivalent to
\[ (f^1 \wedge f^2 - f^3 \wedge f^4, \, 
f^1 \wedge f^3 + f^2 \wedge f^4, \, f^1 \wedge f^4 - f^2 \wedge f^3) \]
(where $f^i$ is a basis for $H^{2k+1}(M;\Z)^\vee$)
even though all Massey products of a space with that cohomology
ring are trivial---
indeed, if $x, y \in H^{2k+1}(M)$ have $x \cup y = 0$ then
$x$ and $y$ are linearly dependent.
\end{ex}

\begin{ex}
Combining Example \ref{ex:rk5} and Theorem \ref{thm:realiseM},
there exists for $k \geq 1$ a non-formal $(2k{-}1)$-connected $(8k{-}1)$-manifold
$M$ with $H^{2k}(M; \Z) \cong \Z^5$ and $H^{4k}(M; \Z) \cong \Z^3$
such that $\csq : \pol^2 H^{2k}(M;\Z) \to H^{4k}(M;\Z)$ is equivalent to
\[ 
2 \, (x_1x_4 + x_3x_5, \;
 x_2x_5 + x_3x_4, \;
 x_1^2 + x_1x_2 + x_2^2 + x_3^2 + x_4^2 + x_5^2) , \]
(describing a homomorphism $\pol^2 \Z^r \to \Z$ as a homogeneous quadratic
polynomial with even cross terms)
even though all Massey products of a space with that cohomology
ring are trivial.
\end{ex}

\subsection{Intrinsic formality}
A space $X$ is said to be \emph{intrinsically formal} if any space with
cohomology algebra isomorphic to $H^*(X)$ is rationally homotopy equivalent
to $X$.
In this case the only defined triple Massey
products are ones such as $\gen{x,y,x}$, which are trivial a priori.
It is immediate from Theorem \ref{thm:obstruction} that any
\tkc{} manifold $M$ of dimension $m \leq 5n{-}3$ whose cohomology algebra has
$\bmspam{M} = 0$ is intrinsically formal, while if $\bmspam{M} \not= 0$
then Theorem \ref{thm:realisation} lets us realise $H^*(M)$ as the
cohomology algebra of some non-formal space (or indeed of a closed manifold).

\begin{cor}
\label{cor:intrinsic}
A closed \tkc{} manifold $M$ of dimension $m \leq 5n{-}3$ is intrinsically formal
if and only if $\bmspam{M} = 0$.
\end{cor}

Cavalcanti \cite[Theorem 4]{cavalcanti06} showed that if $M$ is a closed \tkc{}
\fkm{} and there is an element $\varphi \in H^{2n-1}(M)$ such that
$H^n(M) \to H^{3n-1}(M)$, $x \mapsto \varphi x$ is an isomorphism (``$M$ has a
hard Lefschetz property'') then $M$ is formal if $b_n(M) \leq 2$
and its Massey products vanish uniformly if $b_n(M) \leq 3$. As an
illustration of our results we can make the following improvement.

\begin{thm}
\label{thm:lef}
Let $M$ be a closed \tkc{} \fkm. If $b_n(M) \leq 3$ and there is a 
$\varphi \in H^{2n-1}(M)$ such that $H^n(M) \to H^{3n-1}(M)$,
$x \mapsto \varphi x$ is an isomorphism then $M$ is
intrinsically formal.
\end{thm}

The algebraic claim that this relies on, Proposition \ref{prop:lef}, is
essentially the same as the claim that the Ricci curvature of a manifold of
dimension $\leq 3$ determines the full Riemann curvature tensor.

Closed Kähler manifolds (which covers holonomy groups
$U(n)$ and $SU(n)$) are formal due to the seminal work of Deligne, Griffiths,
Morgan and Sullivan \cite{deligne75}, while Amann and Kapovitch~\cite{amann12}
proved that positive quaternionic-Kähler manifolds (which have holonomy
$Sp(n)Sp(1)$) are always formal.
In contrast, it is an open problem whether Riemannian manifolds with
exceptional holonomy (\ie $G_2$ or $Spin(7)$) need to be formal. 
The hard Lefschetz property holds in particular for closed Riemannian
7-manifolds with holonomy $G_2$, so Theorem \ref{thm:lef} shows that if one
wants to find $G_2$-manifolds that are not formal then one needs to
look for ones with $b_2$ at least 4.
One of our motivations for studying the \bmp{} has been to provide a tool for
testing whether examples of $G_2$-manifolds are formal.


\smallskip
\noindent
{\bf Acknowledgements.} The authors thank Manuel Amann, Jim Davis,
Marisa Fern\'andez, Matthias Kreck, José Moreno-Fernández, Vicente Muñoz and Zhixu Su 
for valuable discussions.
We would also like to thank the referee for several helpful comments and corrections.
DC acknowledges the support of the Leibniz Prize of Wolfgang L\"{u}ck, granted
by the Deutsche Forschungsgemeinschaft.
JN thanks the Simons Foundation for its support under the Simons Collaboration
on Special Holonomy in Geometry, Analysis and Physics
(grant \#488631, Johannes Nordström).

\section{Massey products and the Bianchi-Massey tensor}
\label{sec:massey_products}

The main result of this section that will be used in the other parts of the
paper is that the \bmp{} is well-defined, as claimed in the introduction.
We then put the \bmp{} in context by discussing its relationship with various
notions of triple products: classical Massey, $A_\infty$, and a certain
interpretation of uniform Massey triple products.
In particular, we show that for a DGA whose cohomology satisfies
$m$-dimensional Poincar\'e duality, the top degree part of the \bmp{}
determines the rest of the \bmp, as well as the other notions of triple
products.
While this context is hopefully helpful for understanding the \bmp,
it plays a limited role in the rest of the paper.

\subsection{Some trilinear and quadrilinear algebra}
\label{subsec:quad}

Let us set up some notation for various tensor powers of a 
finite dimensional graded vector space $V$ over $\Q$ and its dual $V^\vee$.

As in the introduction, we let $\gsym^k V$ denote the $k$th graded symmetric
power of $V$, \ie the quotient of $V^{\otimes k}$ by relations of graded
commutativity. Dually, we let 
$\grad^k V^\vee \subseteq (V^\vee)^{\otimes k}$ denote the subspace of
graded symmetric tensors. If $V$ is concentrated in even degree then
$\grad^k V^\vee$ is the same as $\sym^k V^\vee$, the space of symmetric
tensors, while if $V$ is concentrated in odd degree it coincides with the
alternating tensors $\alt^k V^\vee$.

Similarly, let $\agrad^k V^\vee \subseteq (V^\vee)^{\otimes k}$ denote the
subspace of graded antisymmetric tensors. By this we mean that under a
transposition of two arguments, the sign change is the opposite of that for a
graded symmetric tensor.
Let $\agsym^k V$ denote the quotient of $V^{\otimes k}$
by the analogous relations of graded anticommutativity.
So if $V$ is concentrated in even degree then $\agsym^k V = \ext^k V$
and $\agrad^k V^\vee = \alt^k V^\vee$, while if $V$ is concentrated in odd
degree then $\agsym^k V = P^k V$ and $\agrad^k V^\vee = \sym^k V^\vee$.
There are natural dualities between $\grad^k V^\vee$
and $\calg^k V$, and between $\agrad^k V^\vee$ and $\agsym^k V$.

\begin{rmk}
\label{rmk:unnatural}
The projections $\grad^k V \to \gsym^k V$ and $\agrad^k V \to \agsym^k V$
are isomorphisms.
However, that is not true if we replace $V$ by a graded vector space over a
field of arbitrary characteristic or by an abelian group.
Because we will also consider tensors on cohomology groups with integer
coefficients in \S \ref{sec:coboundaries}, we are therefore somewhat fussy
about distinguishing between subspaces and quotients.
\end{rmk}

For any space of tensors where a notion of graded symmetrisation or
antisymmetrisation makes sense, let
\begin{itemize}
\item $\kersym{\bullet}$ denote the kernel of full graded symmetrisation,
\item $\kerasym{\bullet}$ denote the kernel of full graded antisymmetrisation and
\item $\mkern2mu\kercyc{\bullet}$ denote the kernel of graded symmetrisation under
\emph{even} permutations.
\end{itemize}
For example,
$\kersym{\grad^2 V^\vee \otimes V^\vee}$ is the kernel of the symmetrisation
map $\grad^2 V^\vee \otimes V^\vee \to \grad^3 V^\vee$, while
$\kersym{\calg^2 \calg^2 V}$ is the kernel of
$\calg^2\calg^2V \to \calg^4 V$.
Similarly, $\kercyc{(V^\vee)^{\otimes 3}}$ consists of those
$\alpha \in (V^\vee)^{\otimes 3}$ such that
\[ \alpha(x,y,z) + (-1)^{k(i+j)}\alpha(z,x,y) + (-1)^{i(j+k)}\alpha(y, z, x)
= 0 \]
for any $x, y, z \in V$ of degrees $i$, $j$ and $k$ respectively.
Note that
$\kercyc{\bullet}$ contains both $\kersym{\bullet}$ and $\kerasym{\bullet}$.

In the rest of this subsection we study some further features of tensors with
various symmetries, which will be used primarily in
\S\ref{subsec:uniform}--\ref{subsec:ainfty}.
Consider the map
\begin{equation}
\begin{aligned}
\phi :& \, (V^\vee)^{\otimes 3} \to (V^\vee)^{\otimes 3} \\
(\phi \alpha)(x,y,z)
:= & \, (-1)^{jk}\alpha(x,z,y) - (-1)^{ij + jk + ik}\alpha(z,y,x) ,
\end{aligned}
\end{equation}
for $x, y, z$ of degrees $i, j, k$ respectively.
Here are some elementary observations.

\begin{lem}
\label{lem:tri}\hfill
\begin{enumerate}
\item The image of $\phi$ is contained in $\kercyc{(V^\vee)^{\otimes 3}}$,
while $\ker \phi$ consists precisely of those tensors that are graded-invariant
under cyclic permutations.
\item $\third \phi^2$ is a projection onto $\kercyc{(V^\vee)^{\otimes 3}}$.
\item $\phi$ maps $\grad^2 V^\vee \otimes V^\vee$ to
$\agrad^2 V^\vee \otimes V^\vee$ and vice versa, with kernels $\grad^3 V^\vee$
and $\agrad^3 V^\vee$ respectively.
\item $\phi$ maps $\kersym{\grad^2 V^\vee \otimes V^\vee}$ isomorphically to
$\kerasym{\agrad^2 V^\vee \otimes V^\vee}$ and vice versa.
\end{enumerate}
\end{lem}

\begin{proof}\hfill
\begin{enumerate}
\item Trivial.

\item Note that
\[ (\phi^2 \alpha)(x,y,z)
= 2\alpha(x,y,z) - (-1)^{k(i+j)}\alpha(z,x,y) - (-1)^{i(j+k)}\alpha(y,z,x) . \]
Therefore $\phi^2_{|C} = 3\Id$. Meanwhile (i) implies that $\ker \phi$
is a direct complement to $\kercyc{(V^\vee)^{\otimes 3}}$.

\item If $\alpha(y,x,z) = \epsilon(-1)^{ij}\alpha(x,y,z)$
(with $\epsilon = \pm1$), then
\begin{align*}
(\phi\alpha)(y,x,z) &= (-1)^{ik}\alpha(y,z,x) - (-1)^{ij+ik+jk}\alpha(z,x,y) \\
&= \epsilon(-1)^{ik+jk}\alpha(z,y,x) -\epsilon(-1)^{ij+jk}\alpha(x,z,y) =
-\epsilon(-1)^{ij}(\phi\alpha)(x,y,z) .
\end{align*}

\item Follows from (ii) and (iii) and noting that
$\kersym{\grad^2 V^\vee \otimes V^\vee}
= \kercyc{\grad^2 V^\vee \otimes V^\vee}$
and
$\kerasym{\agrad^2 V^\vee \otimes V^\vee}
= \kercyc{\agrad^2 V^\vee \otimes V^\vee}$.
\qedhere
\end{enumerate}
\end{proof}

We can consider $\grad^2 \grad^2 V^\vee$ as the subspace of
$(V^\vee)^{\otimes 4}$ consisting of quadrilinear functions
$\alpha(x,y,z,w)$ that are graded symmetric under swapping
$x \leftrightarrow y$ or $z \leftrightarrow w$,
and under swapping both $x \leftrightarrow z$ and $y \leftrightarrow w$.
Note that
\begin{align*}
(\grad^2V^\vee \otimes (V^\vee)^{\otimes 2}) \; \cap \;
\grad^2 ((V^\vee)^{\otimes 2}) &\; = \; \grad^2\grad^2 V^\vee, \\
(\agrad^2V^\vee \otimes (V^\vee)^{\otimes 2}) \; \cap \;
\grad^2 ((V^\vee)^{\otimes 2}) &\; = \; \grad^2\agrad^2 V^\vee .
\end{align*}
Therefore, if we set
\begin{equation}
\ksmap = \phi \otimes \Id : (V^\vee)^{\otimes 4} \to (V^\vee)^{\otimes 4},
\end{equation}
then Lemma \ref{lem:tri} implies that
$\ksmap$ maps $\grad^2 \grad^2 V^\vee$ to $\grad^2 \agrad^2 V^\vee$
and vice versa. The images are precisely $\kerasym{\grad^2 \agrad^2 V^\vee}$
and $\kersym{\grad^2 \grad^2 V^\vee}$ respectively, as can be seen from
Lemma \ref{lem:tri} and that
\begin{align*}
\kersym{\grad^2 \grad^2 V^\vee} &\; = \; \grad^2 \grad^2 V^\vee
\; \cap \; \kersym{\grad^2 V^\vee \otimes V^\vee} \otimes V^\vee, \\
\kerasym{\grad^2 \agrad^2 V^\vee} &\; = \; \grad^2 \agrad^2 V^\vee
\; \cap \; \kersym{\agrad^2 V^\vee \otimes V^\vee} \otimes V^\vee.
\end{align*}
We obtain a pair of naturally dual exact sequences:
\vspace{-\baselineskip}

\begin{equation}
\label{eq:dualseq}
\xymatrix@R-1.5pc{
0 \ar[r] & \grad^4 V^\vee \ar[r] & \grad^2 \grad^2 V^\vee \ar[r]^\ksmap
& \grad^2 \agrad^2 V^\vee \ar[r] & \agrad^4 V^\vee \ar[r] & 0 \\
0 & \calg^4 V \ar[l] & \calg^2 \calg^2 V \ar[l] & \calg^2 \agsym^2 V
\ar[l]_{\phantom{\vee}\ksmap^\vee}  & \hspace{0pt} \agsym^4 V \ar[l] & 0 \ar[l] }
\end{equation}

\begin{rmk}
\label{rmk:defect}
This shows in particular that there is a natural perfect pairing
\[ \kersym{\gsym^2\gsym^2V} \times \kerasym{\grad^2\agrad^2V^\vee} \to \Q, \]
and $\kersym{\gsym^2\gsym^2V}$ and $\kerasym{\grad^2\agrad^2V^\vee}$ 
can both be regarded as measuring a ``symmetry defect''
of elements of $\grad^2 \grad^2 V^\vee$.
\end{rmk}


\begin{rmk}
\label{rmk:kn}
For $\alpha, \beta \in \grad^2 V^\vee$, we can
define $\alpha \owedge \beta \in \kerasym{\grad^2\agrad^2V^\vee}$ by
\begin{align*}
(\alpha \owedge \beta)(x,y,z,w) \; := \qquad
(-1)^{jk}&
\alpha(x,z)\beta(y,w) - (-1)^{ij+jk+ik}\alpha(z,y)\beta(x,w) \\
-(-1)^{ij+j\ell + k\ell}&\alpha(x,w)\beta(z,y)
+ (-1)^{ij + j\ell + i\ell}\alpha(y,w)\beta(x,z) .
\end{align*}
This induces a linear map
$\grad^2 \grad^2 V^\vee \to \kerasym{\grad^2\agrad^2V^\vee}$, clearly equal to
the restriction of~$\ksmap$.
When $V$ is concentrated in even degree, so that this is a map
$\sym^2 \sym^2 V^\vee \to \kerasym{\sym^2 \alt^2 V^\vee}$,
Besse \cite[Definition 1.110]{besse87} calls $\owedge$ the
\emph{Kulkarni-Nomizu product}.
Some features of this algebra are familiar from the context
of Riemannian geometry, \eg that $\kerasym{\sym^2\alt^2V^\vee}$ has dimension
$\frac{1}{12}(\dim V)^2((\dim V)^2-1)$.
\end{rmk}

\subsection{Well-definedness of the \bmp} 
\label{subsec:bmpdef}

As in the introduction, for a graded ring $H^*$, we let
$\bmspa{H^*} := \kersym{\calg^2 \kc^*} \subseteq \kersym{\gsym^2\gsym^2 H^*}$
where $\kc^*$ the kernel of the product map $\csq \colon \calg^2 H^* \to H^*$,
and abbreviate $\bmspa{H^*(\bullet)}$ to $\bmspa{\bullet}$ when $\bullet$ is
a DGA or a topological space.

In the proof that the \bmp{} is well-defined, we will use the following
notation. For a graded vector space $V^*$, a DGA $\dga$
and linear maps $\rho : V^* \to \dga^{*+r}$, $\sigma : V^* \to \dga^{*+s}$,
let $\rho \symp \sigma : \calg^2 V^* \to \dga$ be the degree $r{+}s$ linear map
induced by the graded commutative (with respect to the grading on $V^*$)
bilinear map
\begin{equation} \label{eq:form_products} \begin{aligned}
V^p \times V^q \to \dga^{p+q+r+s}, \; (v,w) \mapsto & \;
\half \big((-1)^{ps}\rho(v)\sigma(w) + (-1)^{(p+s)q} \rho(w)\sigma(v)\big) \\ 
= & \;
\half \big((-1)^{ps}\rho(v)\sigma(w) + (-1)^{r(p+s)} \sigma(v)\rho(w)\big) .
\end{aligned} \end{equation}
Note that $\rho \symp \sigma = (-1)^{rs} \sigma \symp \rho$, and
the Leibniz rule holds in the sense that
\[ d(\rho \symp \sigma) = (d\rho) \symp \sigma + (-1)^r \rho \symp d\sigma \]
as degree $r{+}s{+}1$ maps $\gsym^2 V \to \dga$.
Further, for $\rho_i$ of degree $r_i$ and $v_i \in V^{q_i}$
\begin{align*}
\big((\rho_1 \symp \rho_2) \symp (\rho_3 \symp \rho_4)\big)
\big((v_1v_2)(v_3v_4)\big)
&= \tfrac{1}{8} \sum_{g \in G} (-1)^{a(g)}\rho_1(v_{g(1)})\rho_2(v_{g(2)})\rho_3(v_{g(3)})\rho_4(v_{g(4)}) \\
&= \tfrac{1}{8} \sum_{g \in G} (-1)^{b(g)}\rho_{g(1)}(v_1)\rho_{g(2)}(v_2)\rho_{g(3)}(v_3)\rho_{g(4)}(v_4),
\end{align*}
where $G \subset S_4$ is the wreath product $S_2 \textrm{ wr } S_2$
(an order 8 subgroup of $S_4$),
and
\begin{align*}
a(g) \; &= \; \sum q_{g(i)} r_j \;+ \sum_{g(i) > g(j)} q_i q_j ,\\
b(g) \; &= \; \sum q_i r_{g(j)} \; + \sum_{g(i) > g(j)} r_i r_j ,
\end{align*}
summing over $1 \leq i < j \leq 4$.
That implies in particular that as long as at most one of the $r_i$ is odd,
\begin{equation}
\label{eq:fullsym}
(\rho_1 \symp \rho_2) \symp (\rho_3 \symp \rho_4)
+ (\rho_1 \symp \rho_3) \symp (\rho_2 \symp \rho_4)
+ (\rho_1 \symp \rho_4) \symp (\rho_2 \symp \rho_3)
\; : \; \gsym^2\gsym^2 V \to \dga
\end{equation}
is in fact fully graded symmetric, \ie it factors through
the graded symmetrisation $\gsym^2 \gsym^2 V \to \gsym^4 V$.

\begin{lem}
\label{lem:principal}
Let $(\dga, d)$ be a DGA, and $\clalg^* := \ker d$.
Choose a right inverse $\rp \colon H^*(\dga) \to \clalg^*$ for the projection
to cohomology, and a linear map $\prp \colon \kc^* \to \dga^{*-1}$ such that
$d\prp = \rp^2$ on $\kc^*$.
Then the linear map
\[ \prp \symp \rp^2 : \calg^2\kc^* \to \dga^{*-1} \]
takes values in $\clalg^{*-1}$ on $\bmspa{\dga} := \bmspa{H^*(\dga)}$,
and the induced map
\[ \princ : \bmsp^*(\dga) \to H^{*-1}(\dga) \]
is independent of the choice of $\prp$ and~$\rp$.
\end{lem}

\begin{proof}
Note that $d(\prp \symp \rp^2) :  \calg^2 \kc^* \to \dga^*$
is the restriction of
$(\rp^2)^2 : \calg^2 \calg^2 H^*(\dga) \to \dga^*$, which factors
through $\calg^4 H^*(\dga)$.
Therefore $d(\prp \symp \rp^2)$ vanishes on the intersection of $\calg^2\kc^*$
with the kernel of
$\calg^2 \calg^2 H^*(\dga) \to \calg^4 H^*(\dga)$, which
is $\bmspa{\dga}$ by definition.
Hence $\prp \symp \rp^2$ maps $\bmspa{\dga} \to \clalg^{*-1}$ as claimed.

Now consider replacing $\prp$ by $\prp_\bullet = \prp + \eta$, for some
$\eta : \kc^* \to \clalg^{*-1}$.
Then $\prp_\bullet \symp \rp^2 - \prp \symp \rp^2 = \eta \symp \rp^2
= -d(\eta \symp \prp)$ takes exact
values on all of $\calg^2 \kc^*$, so certainly $\princ$ is independent of
the choice of~$\prp$, given the choice of $\rp$.

Now consider replacing $\rp$ by $\rp_\bullet := \rp + d\beta$
for some arbitrary linear map $\beta : H^*(\dga) \to \dga^{*-1}$.
Then
\[ \rp_\bullet^2 - \rp^2 
= d\beta \symp (2\rp + d\beta) \]
on $\calg^2 H^*(\dga)$, so a possible choice of $\prp_\bullet$ such that
$d\prp_\bullet = \rp_\bullet^2$ is
\[ \prp_\bullet := \prp + \beta \symp (2\rp + d\beta) . \]
To make the next equation equation unambiguous, we suppress the $\symp$ in the
products that are to be evaluated first.
On $\calg^2 \kc^*$ we have
\begin{align*}
&\prp_\bullet \symp \rp_\bullet^2 - \prp \symp \rp^2 +
d\big(\prp \symp \beta (2\rp + d\beta)
+ \twth \beta\rp \symp \beta d\beta\big) \\
= \; & \beta (2\rp + d\beta) \symp \rp^2
+ \prp \symp (d\beta)(2\rp+d\beta)
+ \beta (2\rp+d\beta) \symp (d\beta)(2\rp+d\beta) \\
& + d\prp \symp \beta(2\rp+d\beta)
- \prp \symp (d\beta)(2\rp+d\beta)
+ \twth (d\beta) \rp \symp \beta d\beta
- \twth \beta \rp \symp (d\beta)^2 \\
= \; & 2\rp^2 \symp \beta (2\rp+d\beta)
+ \beta (2\rp+d\beta) \symp (d\beta)(2\rp+d\beta)
+ \twth \beta d\beta \symp (d\beta)\rp 
- \twth \beta \alpha \symp (d\beta)^2 \\
= \; & 4\rp^2 \symp \beta\alpha
+ 2\rp^2 \symp \beta d\beta
+ 4\beta \rp \symp (d\beta)\rp 
+ \textstyle \frac{8}{3}\beta d\beta \symp (d\beta)\rp 
+ \textstyle \frac{4}{3} \beta \rp \symp (d\beta)^2
+ \beta d\beta \symp (d\beta)^2 .
\end{align*}
The right hand side factors through a map $\calg^4 H^*(\dga) \to \dga^{*-1}$:
the first term, the sum of the second and third terms, the sum of the fourth
and fifth terms, and the sixth term are each of the form \eqref{eq:fullsym},
and hence fully graded-symmetric.
Thus the restriction of $\prp_\bullet \rp_\bullet^2 - \prp\rp^2$ to
$\bmspa{\dga}$ takes exact values,
so $\princ$ is independent of the choice of $\rp$ too.
\end{proof}

\begin{rmk}
\label{rmk:exactpart}
An alternative way to describe the \bmp{} is to argue that the
(unsymmetrised) map $\prp\rp^2 : \kc^* \otimes \kc^* \to \dga^{*-1}, \;
e \otimes e' \mapsto \prp(e)\rp^2(e')$
induces a well-defined map $\kersym{E^* \otimes E^*} \to H^{*-1}(\dga)$.
But $\agrad^2 \kc^* \subseteq \kersym{E^* \otimes E^*}$, and the restriction
of $\prp\rp^2$ to $\agrad^2 \kc^*$ takes exact values.
Thus the induced map descends to $\kersym{E^* \otimes E^*}/\agrad^2 \kc^*$,
which is naturally isomorphic to $\bmspa{\dga}$.
\end{rmk}

\subsection{Uniform Massey triple products}
\label{subsec:uniform}

As before, let $\dga^*$ be a DGA, and $\clalg^*$ the subspace of closed
elements.
Let $H^*$ denote the cohomology of $\dga^*$, and $\kc^* \subseteq \gsym^2 H^*$
the kernel of the product $\gsym^2 H^* \to H^*$.

Consider $\tripsp$, \ie the kernel of the full
symmetrisation $\kc^* \otimes H^* \to \calg^3 H^*$.
Given choices of $\alpha : H^* \to \clalg^*$ and
$\gamma : \kc^* \to \dga^{*-1}$ such that $d\gamma = \alpha^2$ as before, we
get an induced map $\gamma \alpha : \kc^* \otimes H^* \to \dga^{*-1}$ whose
restriction to $\tripsp$ takes closed values. We call the resulting map
\enlargethispage{0.9\baselineskip}
\begin{equation}
\label{eq:uniform}
\trip : \tripsp^* \to H^{*-1}
\end{equation}
the \emph{uniform Massey triple product} of $\dga$. Unlike the Bianchi-Massey
tensor, it does depend on the choices of $\alpha$ and $\gamma$, in the
following way.

Given different choices $\alpha', \gamma'$, pick $\beta : H^* \to \dga^{*-1}$
such that $d\beta = \alpha' - \alpha$.
Then $\gamma' - \gamma - \beta(\alpha'+\alpha)$ maps $\kc^* \to \clalg^{*-1}$.
Let $\delta_\beta$ be the induced map $\kc^* \to H^{*-1}$ (which does depend on
the choice of $\beta$, but in a way that is not significant). Now
\[ \trip' - \trip : \tripsp \to H^{*-1} \]
is the restriction of the map
$\delta_\beta \Id : \kc^* \otimes H^* \to H^{*-1},
\; e \otimes x \mapsto \delta_\beta(e)x$ (for any choice of $\beta$).
Thus $\trip$ is well-defined precisely modulo restrictions to $\tripsp$
of maps of the form $\delta \Id : \kc^* \times H^* \to H^{*-1}$
for $\delta : \kc^* \to H^{*-1}$.

Next we prove that $\trip$ is equivalent to $\princ$ for Poincar\'e DGAs.

\begin{defn}
\label{def:poincare}
Let $H^*$ be a finite dimensional graded commutative algebra over $\Q$.
We call $\alpha \in (H^m)^\vee$ a \emph{Poincar\'e class} if
the linear map
\[ \alpha \cap : H^i \to (H^{m-i})^\vee, \; x \mapsto (y \mapsto \alpha(xy)) \]
is an isomorphism for all $i$. We say that $H^*$ is
{\em $m$-dimensional Poincar\'e} if such an $\alpha$ exists.

A DGA $\dga$ is $m$-dimensional Poincar\'e if its cohomology is.
\end{defn}

\begin{lem}
\label{lem:unitrip_to_bm}
\begin{enumerate}
\item For any DGA $\dga$, $\trip$ determines $\princ$.
\item If $\dga$ is $m$-dimensional Poincar\'e, then the top component
$\princ : \bmspam{\dga} \to H^m(\dga)$ determines~$\trip$ (and hence also
the other components of $\princ$).
\end{enumerate}
\end{lem}

\begin{proof}
We can control all the relevant maps in terms of
\[ \mu : \kersym{\kc^* \otimes H^* \otimes H^*} \to H^{*-1} \]
defined by restricting $\kc^* \otimes H^* \otimes H^* \to \dga^{*-1}, \; e \otimes x \otimes y \mapsto \gamma(e)\alpha(x)\alpha(y)$ to
$\kersym{\kc^* \otimes H^* \otimes H^*}$ and observing that the restriction
takes values in $\clalg^{*-1}$. 
$\kersym{\kc \otimes H^* \otimes H^*}$ contains $\tripsp \otimes H^*$, and the
restriction of $\mu$ to that equals $\trip \Id$.
This is determined by $\trip$, and if $\dga$
is $m$-dimensional Poincar\'e then the degree $m$ component of
$\trip \Id$ also determines $\trip$.

The map
$\mu$ factors through the projection
$\kersym{\kc^* \otimes H^* \otimes H^*} \to \kersym{\kc^* \otimes \calg^2 H^*}$.
The restriction to $\kersym{\kc^* \otimes \kc^*}$ in turn factors through
the projection to $\bmspa{H^*}$, where it induces $\princ$
(see Remark~\ref{rmk:exactpart}).

Now observe that the Bianchi identity for elements of
$\kersym{\grad^2 \calg^2 H^*} \subset \kersym{\calg^2 H^* \otimes \calg^2 H^*}$
means that this space is contained in the image of
$\kersym{\calg^2 H^* \otimes H^*} \otimes H^*$. Hence the image of
$\tripsp \otimes H^*$ in $\kersym{\kc^* \otimes \calg^2 H^*}$ contains
$\kersym{\grad^2 \kc^*}$, which maps onto $\bmspa{H^*}$. This proves~(i).

For (ii), we need to prove that if $\princ : \bmspam{\dga} \to H^m(\dga)$
is trivial, then the choices in the definition of $\trip$ can be made so
that $\trip$ vanishes too. It suffices to show that if the restriction
of $\mu$ to the degree $m{+}1$ part of $\kersym{\kc^* \otimes \kc^*}$ vanishes,
then we can make the choices so that $\mu$ vanishes on all of the degree
$m{+}1$ part of $\kersym{\kc^* \otimes \calg^2 H^*}$.
Now let $D^* \subset \calg^2 H^*$ be a direct complement to $\kc^*$. Then
$\kc^* \otimes \calg^2 H^*
= \kc^* \otimes \kc^* \; \oplus \; \kc^* \otimes D^*$.
Let $p : \kc^* \otimes \calg^2 H^* \to \kc^* \otimes D^*$ be the corresponding
projection.

Now, if we change the choice of $\gamma$ to $\gamma' = \gamma + \eta$
for some $\eta : \kc^* \to \clalg^{*-1}$, then the map
$\mu' - \mu : \kersym{\kc^* \otimes \calg^2 H^*} \to H^{*-1}$ is the restriction of
the degree $-1$ map $\kc^* \otimes \calg^2 H^* \to H^{*-1}$ induced by
$[\eta] \times \csq$. The latter vanishes on $\kc^* \otimes \kc^*$,
while by Poincar\'e duality any map from the degree $m{+}1$ part of
$\kc^* \otimes D^*$ to $H^m(\dga)$ can realised this way for some
choice of $\eta$.

Since the kernel of the restriction
$p : \kersym{\kc^* \otimes \calg^2 H^*} \to \kc^* \otimes D^*$ is precisely
$\kersym{\kc^* \otimes \kc^*}$, it follows that by changing the choice of $\gamma$,
we can ensure that $\mu$ is zero on the degree $m+1$ part of
$\kersym{\kc^* \otimes \calg^2 H^*}$ provided that the restriction
to $\kersym{\kc^* \otimes \kc^*}$ vanishes.
\end{proof}

\subsection{Massey triple products}
\label{subsec:ordinary}

Let us recall the definition of Massey triple products.
Let $(\dga, d)$ be a DGA, and $\clalg := \ker d$.
Suppose $x \in H^i(\dga)$, $y \in H^j(\dga)$ and $z \in H^k(\dga)$,
such that $xy = yz = 0$.
Choose representatives $\alpha_x, \alpha_y, \alpha_z \in \clalg^*$.
Then $\alpha_x\alpha_y$ and $\alpha_y\alpha_z$ are exact, say $d\gamma_{xy}$ and
$d\gamma_{yz}$ respectively. Then
$\gamma_{xy} \alpha_z - (-1)^i\alpha_x\gamma_{yz} \in \dga^{i+j+k-1}$
is closed, so represents a class in $H^{i+j+k-1}(\dga)$.
The choices of $\alpha$s and $\gamma$s can change that class by elements
of $xH^{j+k-1}(\dga) + zH^{i+j-1}(\dga)$, but the image
\[ \gen{x, y, z} \; \in \; \frac{H^{i+j+k-1}(\dga)}
{xH^{j+k-1}(\dga) + zH^{i+j-1}(\dga)} \]
is well-defined, and that is the Massey triple product.

Note that $(xy)z - (-1)^{i(j+k)}(yz)x \in \kersym{\kc^* \otimes H^*}$,
and is mapped to $\gen{x, y, z}$ by the uniform triple product $\trip$ from
\eqref{eq:uniform}.
Thus the Massey triple products are determined by $\trip$, as one would expect.
By Lemma \ref{lem:unitrip_to_bm}, the Massey triple products of
a Poincar\'e duality DGA are thus also determined by the top component
of the Bianchi-Massey tensor $\princ$.
However, it may be illuminating to consider how to recover the Massey triple
products from $\princ$ more directly.

If $x, y, z, w \in H^*(\dga)$ of degrees $i, j, k$ and $\ell$ respectively
and $xz = yz = xw = yw = 0 \in H^*(\dga)$, then
\begin{equation}
\label{eq:dga_tensor}
\mq(x,y,z,w) :=
(-1)^{jk} \gen{x,z,y}w \in H^{i+j+k+\ell-1}(\dga) 
\end{equation}
is well-defined. This approach to eliminating the indeterminacy of the triple
products has been exploited by Taylor \cite{taylor10}.

The next elementary lemma lets us recover the defined values
of \eqref{eq:dga_tensor} from the \bmp.
Write the product $H^* \times H^* \to \calg^2 H^*$
as $(x,y) \mapsto x \esymp y$.

\begin{lem}
\label{lem:recover}
If $x,y,z,w \in H^*(\dga)$ such that $xz = xw = yz = yw \in \kc^*$
(\ie the products in $H^*(\dga)$ vanish) then
\begin{equation}
\label{eq:nongen}
v := (x \esymp z) \poltp (y \esymp w) -
(-1)^{jk+j\ell +k\ell}(x \esymp w) \poltp (y \esymp z) \in \bmspa{\dga} ,
\end{equation}
and
\begin{equation}
\label{eq:recover}
\princ(v) = \gen{x,z,y}w .
\end{equation}
\end{lem}

\begin{defn}
\label{def:ordinary}
We call elements of $\bmspa{\dga}$ of the form \eqref{eq:nongen}
\emph{ordinary}.
\end{defn}

\begin{lem}[{\cf Hepworth \cite[Lemma 3.1.4]{hepworth05}}]
\label{lem:syms}
\hfill
\begin{enumerate}
\item
If $x, y, z, w \in H^*(\dga)$ such that $xz = yz = xw = yw = 0 \in H^*(\dga)$,
then
\begin{align*}
\mq(x,y,z,w)\;  &= \; -(-1)^{ij}\mq(y,x,z,w) \\
= -(-1)^{k\ell}\mq(x,y,w,z)\; &= \; (-1)^{(i+j)(k+\ell)} \mq(z,w,x,y)
\end{align*}
\item If in addition $xy = zw = 0 \in H^*(\dga)$ then
\[ \mq(x,y,z,w) + (-1)^{k(i+j)} \mq(z,x,y,w) + (-1)^{i(j+k)}\mq(y,z,x,w) = 0. \]
\end{enumerate}
\end{lem}

If the product $H^*(\dga) \times H^*(\dga) \to H^*(\dga)$ is
trivial in non-zero degrees then $\mq$ induces
a linear map $H^{>0}(\dga)^{\otimes 4} \to H^{*-1}(\dga)$.
Equivalently, if $\alpha \in H^m(\dga)^\vee$, then $\alpha \circ \mq$
is in the degree $m{+}1$ part of $(H^{>0}(\dga)^\vee)^{\otimes 4}$.
Lemma \ref{lem:syms}(i) means that $\alpha \circ \mq$ is graded anti-symmetric
under swapping $x \leftrightarrow y$ or $z \leftrightarrow w$, and also
symmetric under swapping both $x \leftrightarrow z$ and $y \leftrightarrow w$,
so $\alpha \circ \mq \in \grad^2 \agrad^2 H^*(\dga)^\vee$.
Moreover, (ii) says that the Bianchi identity holds, so $\alpha \circ \mq$
in fact belongs to the degree $m{+}1$ part of
$\kerasym{\grad^2\agrad^2 H^{>0}(\dga)^\vee}$.

Now suppose that $\dga$ is $m$-dimensional Poincar\'e (\eg that $\dga$ is the
DGA of piecewise linear forms on a closed oriented $m$-manifold $M$).
Then for $x \in H^i(\dga)$, $y \in H^j(\dga)$ the
annihilator of
$xH^{j+k-1}(\dga) + yH^{i+k-1}(\dga) \subseteq H^{i+j+k-1}(\dga)$
is precisely
\[ \{ w \in H^{m+1-i-j-k}(\dga) : xw = yw = 0 \in H^*(\dga) \} . \]
Hence for $x \in H^i(\dga)$, $y \in H^j(\dga)$, $z \in H^k(\dga)$ such that
$xz = yz = 0$ the triple product $\gen{x,z,y}$ is completely
determined by the values of $\mq(x,y,z,w) \in H^m(\dga) \cong \Q$ for
$w \in H^{m+1-i-j-k}(\dga)$ such that $xw = yw = 0$, and hence by the \bmp.

On the other hand, suppose that $N \subset H^*(\dga)$ is a subspace
such that the product $N \times N \to H^*(\dga)$ is trivial
(so $\gsym^2 N \subseteq E$) and that $\bmspam{\dga}$ is the degree $m{+}1$ part
of $\kersym{\gsym^2\gsym^2 N}$;
\eg if $\dga$ is \tkc{} with $H^n(\dga) \times H^n(\dga) \to H^{2n}(\dga)$
trivial and $m = 4n{-}1$ then we could take $N = H^n(\dga)$.
Then \eqref{eq:recover} means that $\alpha \circ \princ$ can be recovered from
$\alpha \circ \mq$ using the duality
$\kerasym{\grad^2\agrad^2 N^\vee} \cong \kersym{\gsym^2\gsym^2 N}^\vee$
from Remark \ref{rmk:defect}.
In particular, if $\dga$ is in addition $m$-dimensional Poincar\'e, then
$q$ determines the top component of $\princ$, and hence the rest of
$\princ$ too. Put differently, the surjectivity of the map
$\ksmap^\vee : \calg^2 \agsym^2 N \to \kersym{\gsym^2\gsym^2 N}$
implies that in this situation $\bmspam{\dga}$ is spanned by ordinary elements;
meanwhile \eqref{eq:recover} implies that the top component of $\princ$ can be
recovered from the Massey triple products whenever $\bmspam{\dga}$ is generated
by ordinary elements.

We will see in \S \ref{subsec:nonordinary} that when the product structure of
$H^*(\dga)$ is non-trivial, then it is often \emph{not} the case that
$\bmspam{\dga}$ is generated by ordinary elements. Then $\princ$ is not
determined by Massey triple products.

\subsection{Relationship with
\texorpdfstring{$A_\infty$}{A-infinity}-structures}
\label{subsec:ainfty}

For any DGA $\dga$, one may define an $A_\infty$-structure on $H^*(\dga)$,
which consists of a sequence of linear maps
$\mu_k \colon H^*(\dga)^{\otimes k} \to H^*(\dga)$ of degree $2{-}k$, for
$k \geq 2$, see \eg Amann \cite[\S 8.5]{amann15} or Vallette \cite{vallette12}.
$\mu_2$ is simply the product on the cohomology algebra.
The definition of the higher products relies on arbitrary choices, but the
structure is well-defined up to a suitable notion of $A_\infty$-isomorphism.
Moreover, $H^*(\dga)$ with this $A_\infty$-structure is quasi-isomorphic to 
$\dga$ itself, so in particular determines the homotopy type of $\dga$.
There is also a notion of homotopy equivalence of $A_\infty$ algebras, and two
simply-connected spaces are rationally homotopy equivalent if and only if their
cohomology rings are equivalent in
the sense of Kadeishvili \cite[Theorem 9.1]{kadeishvili09};
see also \mbox{Vallette \cite[Theorem 8]{vallette12}}
or Amann \cite[\S 8.5]{amann15}.

We shall only be concerned with
$\mu_3 \colon H^*(\dga)^{\otimes 3} \to H^{*-1}(\dga)$, which may be defined
as follows. Let $\clalg^* \subseteq \dga^*$ denote the subspace of closed
elements as before. Pick a right inverse $\alpha \colon H^*(\dga) \to \clalg^*$
of the projection $\clalg^* \to H^*(\dga)$, and
a $\gamma \colon \clalg^* \to \dga^{*-1}$ such that
$d \gamma \colon \clalg^* \to \clalg^*$ is a projection to the exact part.
Further pick a map $p : \dga^* \to H^*(\dga)$ such that $p(\beta) = [\beta]$
for $\beta \in \clalg^*$, and set
\[ \mu_3(x,z,y) := p\big(\gamma(\alpha(x)\alpha(z))\alpha(y)
- (-1)^i \alpha(x)\gamma(\alpha(z)\alpha(y))\big) \in H^*(\dga) . \]
If $xz = yz = 0$ then clearly
\[ \mu_3(x,z,y) = \gen{x,z,y} \mod xH^{j+k-1}(\dga) + yH^{i+k-1}(\dga) , \]
\cf \cite[Lemma 9.4.6]{loday12}.
(For $k \geq 4$, the precise relationship between $k$-fold Massey products and
the higher products $\mu_k$ is more subtle: see Buijs, Moreno-Fernández and
Murillo \cite{bmm18}.)

We can also relate $\mu_3$ to the uniform Massey triple product.
If we define $\wh \mu_3 : \calg^2 H^* \otimes H^* \to H^{*-1}$ to be
induced by $(x,y,z) \mapsto p(\gamma(\alpha(x)\alpha(y))\alpha(z)$, then 
\[ (\phi^\vee \circ \wh \mu_3)(x,y,z)
= \wh \mu_3((-1)^{jk}(xz)y - (-1)^{ij+ik+jk}(zy)x) = (-1)^{jk}\mu_3(x,z,y) , \]
so $\wh \mu_3$ determines $\mu_3$.
Meanwhile, the restriction of $\wh \mu_3$ to
$\kersym{E^* \otimes H^*}$ is the uniform triple product~$\trip$.
Because $\mu_3$ determines the restriction of $\wh \mu_3$ to
$\kersym{\calg^2 H^* \otimes H^*}$ by the duality of Lemma \ref{lem:tri},
it also determines $\trip$.

The next lemma is essentially a converse statement.

\begin{lem}
Suppose $\dga$ and $\dga'$ are DGAs with a chosen isomorphism
$H^*(\dga) \cong H^*(\dga')$, such that $\trip = \trip'$.
Then there is a DGA $\cale$ with $H^*(\cale) = 0$ such that the choices in the
definition of the $A_\infty$ triple products $\mu_3$ of $\dga \times \cale$
and $\mu'_3$ of $\dga'$ can be made so that $\mu_3 = \mu'_3$.
\end{lem}

\begin{proof}
The hypothesis means that, having made the choices of $\alpha'$, $\gamma'$
and $p'$ in the definition of $\mu'_3$, we can choose $\alpha$ and $\gamma$
such that the restrictions of $\wh \mu_3$ and $\wh \mu'_3$ to
$\kersym{E^* \otimes H^*}$ agree.

Now, take $\cale$ to be the free DGA (\ie the differential of any generator
is never contained in the multiplicative algebra of the generators) whose
generators in degree $k$ is the degree $k{+}1$ part of the image of
$\csq : \calg^2 H^* \to H^*$. Choose
$\gamma : \calg^2 H^* \to (\dga \times \cale)^{*-1}$ to be the sum of the
$\gamma : \calg^2 H^* \to \dga^{*-1}$ above and
$\csq : \calg^2 H^* \to \cale^{*-1}$.
Then $d(\alpha \gamma) : \calg^2 H^* \otimes H^* \to \dga^* \times \cale^*$
has no kernel outside $\kersym{\calg^2 H^* \otimes H^*}$. Therefore,
by choosing $p$ we can adjust the map
$\wh \mu_3 : \kersym{\calg^2 H^* \otimes H^*} \to H^{*-1}$ arbitrarily subject
only to its restriction to $\kersym{E^* \otimes H^*}$ equalling $\trip$.
\end{proof}

In summary, $\mu_3$ determines $\trip$, and hence also $\princ$.
On the other hand, $\trip$ essentially captures all the information of $\mu_3$,
but in a way that makes the dependence on choices more transparent.
When $\dga$ satisfies Poincar\'e duality, $\princ$ in turn captures all the
information of $\trip$ in a way that eliminates dependence on choices
altogether.

\section{The rational homotopy type} \label{sec:QHT}

In this section we prove our main theorems on the significance of the \bmp{}
for the rational homotopy type---and hence formality and diffeomorphism
classification up to finite ambiguity---of \tkc{} manifolds of dimension
$m \leq 5n{-}3$.

\subsection{Minimal Sullivan algebras}
\label{subsec:minalg}

We first classify $(n{-}1)$-connected $m$-dimensional
Poincar\'e minimal Sullivan algebras via their cohomology algebras and \bmp s.
Recall that a \emph{minimal Sullivan algebra} is a DGA $(\mnalg, d)$ that is
free as a graded algebra, $\mnalg \cong \Lambda V$, and has a well-ordered
basis $\{v_\alpha\} \subset V$ such that $dv_\alpha$ lies in the
subalgebra generated by $\{ v_\beta : \beta < \alpha\}$, and
$\alpha \leq \beta \Rightarrow \deg v_\alpha \leq \deg v_\beta$.
We are only interested in the case when $\mnalg$ is simply-connected.
In this case, the minimality condition
reduces to saying that $\mnalg$ is free, and
\begin{equation}
\label{eq:min}
\textrm{for any $v \in \mnalg^i$, 
$dv$ is a linear combination of products of elements of degree} < i .
\end{equation}

Let $\dga$ be a finite-dimensional graded commutative algebra or DGA
over $\Q$.
We call $\dga$ \mbox{\em $j$-connected} if $\dga^0 = \Q$ and $\dga^k = 0$ for 
$1 \leq k \leq j$. Recall also from Definition \ref{def:poincare} that a
DGA $\dga$ is $m$-dimensional Poincar\'e if and only if $H^*(\dga)$ is.

The key to the role of the \bmp{} in this paper is the following existence and
uniqueness result for minimal Sullivan algebras with prescribed \bmp.

\begin{thm}
\label{thm:minalg}
Let $n \geq 2$.
\begin{enumerate}
\item \label{it:minexist}
Let $m \leq 5n{-}2$.
For every \tkc{} $m$-dimensional Poincar\'e algebra $H^*$ (in the sense of
Definition \ref{def:poincare}) 
and linear map $\princ \colon \bmspam{H^*} \to H^m$, 
there exists an $(n{-}1)$-connected minimal Sullivan algebra $\mnalg$ with 
$H^*(\mnalg) = H^*$ and \bmp~$\princ$.

\item \label{it:minunique}
Let $m \leq 5n{-}3$.
If $\mnalg_1$ and $\mnalg_2$ are \tkc{} $m$-dimensional Poincar\'e
minimal Sullivan algebras 
and
$G \colon H^*(\mnalg_1) \to H^*(\mnalg_2)$ is an isomorphism of the cohomology
algebras then there is a DGA isomorphism $\phi \colon \mnalg_1 \to \mnalg_2$
such that $\phi_\# = G$ if and only if the diagram below commutes.
\[ \xymatrix{
\bmspam{\mnalg_1} \ar[d]^{\princ_1} \ar[r]^{G}
& \bmspam{\mnalg_2} \ar[d]^{\princ_2} \\
H^m(\mnalg_1) \ar[r]^{G} & H^m(\mnalg_2)}  \]
\end{enumerate}
\end{thm}

Let us first outline the essence of the proof.
The standard technique is to describe the
degree~$i$ part $V^i$ of the generating set of the minimal algebra
$\mnalg = \Lambda V$ recursively in terms of the cohomology algebra and any
further data (\ie the \bmp{} in this case), using the property \eqref{eq:min}.
Let $\mnalg_\trc{i}$ denote the sub-DGA generated by elements of
degree $\leq i$, or equivalently $\mnalg_\trc{i} = \Lambda V^{\leq i}$. 
Presenting $V^i$ as the sum of its closed subspace $C^i$ and a direct
complement~$N^i$, in the $i$th step of the recursion one identifies
$N^{i-1}$ and $C^i$ with the kernel and cokernel of
$H^i(\mnalg_\trc{i-2}) \to H^i(\mnalg)$ respectively
(this relies on our algebras being simply-connected, so that $V^1 = 0$).
In the setting of Theorem \ref{thm:minalg},
this map is essentially determined by the cohomology algebra except for 
$3n{-}1 \leq i \leq m{-}n$ and $i = m$, where the \bmp{} appears (actually, it
might be more accurate to say that the uniform triple product $\trip$ appears,
and that this is controlled by the \bmp{} via Lemma \ref{lem:unitrip_to_bm}). 
To prove the uniqueness statement \ref{it:minunique}, one can argue that
the generating set of any minimal algebra with the prescribed cohomology
and \bmp{} can be described this way---while the description involves some
arbitrary choices, those can be expressed in terms of the cohomology algebra. 
(If $m = 5n{-}2$, then some information about fourfold Massey products would
be needed to capture the essence of $H^i(\mnalg_\trc{i-2}) \to H^i(\mnalg)$ for
$i = m{-}n$ and $m$, which is why \ref{it:minunique} requires $m \leq 5n{-}3$.)

We will split the argument into two parts. First we explain in
Proposition \ref{prop:PD} that more generally, a minimal algebra $\mnalg$ that is $m$-dimensional Poincar\'e is
essentially characterised by $\mnalg_\trc{k}$
for any $k$ with $2k \geq m{-}1$, together with
the map $H^m(\mnalg_\trc{k}) \to H^m(\mnalg)$. This is implicit in
Kreck-Triantafillou \mbox{\cite[Theorem 1.2]{kreck91a}}.
Proposition \ref{prop:trunc} 
then shows that for \tkc{} 
Poincar\'e minimal algebras of dimension $m \leq 5n{-}3$,
that data is characterised by the cohomology algebra and the \bmp.

Proving Theorem \ref{thm:minalg} in one go would be shorter in total
(though longer than the proofs of Propositions \ref{prop:PD} and
\ref{prop:trunc} individually), in part because splitting up the proof as we do
involves first encoding much of the algebra structure of $H^*(\mnalg)$ in terms
of the map $H^m(\mnalg_\trc{k}) \to H^m(\mnalg)$, and then reconstructing the
algebra again. 
Our reason for organising the proof as we do is that we hope that it will make
is easier to identify what further invariants are needed to determine the
rational homotopy type if we weaken the connectedness hypothesis,
and also that it clarifies the relation to the results of Kreck and
Triantafillou.

\subsection{Reconstructing a minimal DGA from a partial Poincar\'e class}

We employ the following terminology from
Kreck and Triantafillou \cite[\S 1]{kreck91a}.

\begin{defn} \label{def:pp}
Let $H^*$ be a finite dimensional graded commutative algebra over the
rationals.
For $2k{+}1 \geq m$, we call $\alpha \in (H^m)^\vee$ a
\emph{$k$-partial Poincar\'e class} if the map
\[ \alpha \cap : H^i \to (H^{m-i})^\vee, \; x \mapsto (y \mapsto \alpha(xy)) \]
is an isomorphism for $m - k \leq i \leq k$ and injective for
$i = k{+}1$ (and hence surjective for $i = m{-}k{+}1$).
\end{defn}

We aim to prove in Proposition \ref{prop:PD} that a minimal DGA
$\mnalg$ that is $m$-dimensional Poincar\'e
can essentially be reconstructed from its
truncation $\mnalg_\trc{k}$ and a $k$-partial Poincar\'e class
$\alpha \in H^m(\mnalg_\trc{k})$, 
provided $2 \leq k \leq 2m{+}1$.
Let us first consider the following easier problem of reconstructing a Poincar\'e
algebra $\bar H^*$ from a truncation $\bar H^{\leq k+1}$ together with a
$k$-partial Poincar\'e class in the degree $m$ part of
$\Lambda \bar H^{\leq k+1}$ (by which we mean the algebra generated by
$\bar H^{\leq k+1}$, with no relations imposed beyond those generated by the
relations of $\bar H$ in degree $\leq k{+}1$).

\begin{lem}
\label{lem:partial}
Let $k \geq 2$ and $m \leq 2k{+}1$.
\begin{enumerate}
\item \label{it:restrict}
Let $\phi : H^* \to \bar H^*$ be an algebra homomorphism that is an isomorphism
in degree $\leq k$, and let $\bar \alpha \in \bar H^*$ be a Poincar\'e class.
Then $\phi^\vee \bar \alpha \in H^m$ is a $k$-partial Poincar\'e class
if and only if $\phi$ is injective.

\item \label{it:reconstr}
Let $H^*$ be an algebra generated by elements of degree $\leq k{+}1$,
and let $\alpha \in H^m$ be a $k$-partial Poincar\'e class.
Then there exists a unique (up to isomorphism)
$m$-dimensional Poincar\'e duality algebra
$\bar H^*$ with an algebra homomorphism $\phi : H^* \to \bar H^*$ and
Poincar\'e class $\bar \alpha \in (\bar H^m)^\vee$ such that
\begin{itemize}
\item $\phi : H^i \cong \bar H^i$ is an isomorphism for $i \leq k$ and
$\phi : H^{k+1} \into \bar H^{k+1}$ is injective;
\item $\phi^\vee \bar \alpha = \alpha$.
\end{itemize}
\end{enumerate}
\end{lem}

\begin{proof}
\ref{it:restrict} Immediate from the commutativity of the diagram
\[ \xymatrix{
H^i \ar[d]_{(\phi^\vee \bar\alpha) \cap} \ar[r]^{\phi}
& \bar H^i \ar[d]^{\bar\alpha\cap} \\
(H^{m-i})^\vee & (\bar H^{m-i})^\vee.
\ar[l]_(0.425){\phi^\vee} } \]
\ref{it:reconstr} Let $\bar H^i := H^i$ for $i \leq k$, and
$\bar H^i := (H^{m-i})^\vee$ for $i > k$. Define the product of
$x \in \bar H^i$ and $y \in \bar H^j$ using the product structure of $H^*$
if $i + j \leq k$, and
as $\bar \alpha \cap (xy) \in (H^{m-i-j})^\vee = \bar H^{i+j}$
if $ i, j \leq k$ and $i+j \geq k$.
If $i \leq k$ and $j > k$, so that $y \in (H^{m-j})^\vee$, let
$xy = y \cap x \in (H^{m-i-j})^\vee$, \ie the map
$H^{m-i-j} \to \Q, z \mapsto y(xz)$.
If $i, j > k$ then $xy = 0$.
\end{proof}

Given an $m$-dimensional Poincar\'e algebra
$(\bar H^*, \bar \alpha)$, the multiplication induces a natural map
$\phi : \Lambda \bar H^{\leq k+1} \to \bar H^*$ that is an isomorphism
in degrees $\leq k{+}1$. Applying Lemma \ref{lem:partial}\ref{it:reconstr} to
this $\phi$ and $\alpha := \phi^\vee \bar \alpha$ one recovers the original
$\bar H^*$.

On the other hand, given an algebra $H^*$, let $H_\trc{k}^*$ denote the
subalgebra generated by classes of degree $\leq k$
(the image of the natural map $\Lambda H^{\leq k} \to H^*$).
Note that whether $\alpha \in (H^m)^\vee$ is a $k$-partial Poincar\'e class
depends only on the restriction of $\alpha$ to $H_\trc{k+1}^m$.
For any $k$-partial Poincar\'e class $\alpha \in (H^m)^\vee$, we can thus apply
Lemma \ref{lem:partial}\ref{it:reconstr} to
$\bigl(H_\trc{k+1}^*, \; \alpha_{|H_\trc{k+1}^m} \bigr)$ 
to canonically
construct a Poincar\'e algebra $\bar H^*$ (though there need not be a canonical
way to extend the homomorphism $H_\trc{k+1}^* \to \bar H^*$ to $H^*$).

In this sense,
in Proposition \ref{prop:PD}\ref{it:PD_exist} the restriction of
$\alpha$ to $H^m_\trc{k+1}(\tdga)$
is responsible
for reconstructing the cohomology algebra of $\mnalg$, while the remaining
components of $\alpha$ encode Massey products etc.

Note that Lemma \ref{lem:partial}\ref{it:restrict} implies in particular that
if $\mnalg$ is a DGA and $\alpha \in H^m(\mnalg)^\vee$ is a Poincar\'e class
then the image of $\alpha$ in $H^m(\mnalg_\trc{k})^\vee$ is a $k$-partial
Poincar\'e class (where as before $\mnalg_\trc{k} \subseteq \mnalg$ denotes the
sub-DGA generated by elements of degree $\leq k$).

\begin{prop}
\label{prop:PD}
Let $k \geq 2$ and $m \leq 2k{+}1$.
\begin{enumerate}

\item \label{it:PD_exist}
Let $\tdga$ be a simply-connected minimal Sullivan algebra generated in
degree $\leq k$, and let $\alpha \in H^m(\tdga)^\vee$ be a $k$-partial
Poincar\'e class. Then there is a minimal Sullivan algebra $\mnalg$ with
Poincar\'e class $\alpha_\mnalg \in H^m(\mnalg)^*$ and an isomorphism
$\tau \colon \tdga \to \mnalg_\trc{k}$ such that
$\tau_\#^\vee\alpha_\mnalg = \alpha$.

\item \label{it:PD_unique}
Let $\mnalg_1$, $\mnalg_2$ be simply-connected minimal Sullivan algebras
that are $m$-dimensional Poincar\'e. 
Let $\tau \colon \mnalg_{1 \trc{k}} \to \mnalg_{2 \trc{k}}$ and
$G \colon H^m(\mnalg_1) \to H^m(\mnalg_2)$ be isomorphisms,
such that the diagram below commutes.
\[ \xymatrix{
H^m(\mnalg_{1 \trc{k}}) \ar[d] \ar[r]^{\tau_\#}  & H^m(\mnalg_{2 \trc{k}})
\ar[d] \\
H^m(\mnalg_1) \ar[r]^{G} & H^m(\mnalg_2) } \]
Then there is an isomorphism $\phi \colon \mnalg_1 \to \mnalg_2$
such that the restriction
$\phi_\trc{k} \colon \mnalg_{1 \trc{k}} \to \mnalg_{2 \trc{k}}$ equals $\tau$
and $\phi_\# \colon H^m(\mnalg_1) \to H^m(\mnalg_2)$ equals $G$.
\end{enumerate}
\end{prop}

\begin{proof}
\ref{it:PD_exist}
If $k \geq m$, then constructing the generators $V^i$ in degree $i > k$ for
$\mnalg$ is a trivial recursion: $d$~maps $V^i = N^i$ isomorphically to
the closed subspace of $\mnalg_\trc{i}^{i+1}$.
The following claim, which lets us increase $k$ inductively until we reach
$k = m$, therefore proves the result.

\smallskip
{\narrower\noindent
There exists a minimal Sullivan algebra $\edga$ generated in degree $\leq k{+}1$
with a $(k{+}1)$-partial Poincar\'e class $\alpha_\edga \in H^m(\edga)^\vee$ and
an isomorphism $\phi \colon \tdga \to \edga_\trc{k}$ such that
$\phi_\#^\vee \alpha_\edga = \alpha$.\par}

\smallskip\noindent
Let us first construct the generating space $V$ for $\edga$.
We take the degree $\leq k$ parts to equal those of~$\tdga$
(and define $\phi$ to be the inclusion).
Choose a direct complement $C^{k+1}$ of the image of
$\alpha\cap : H^{k+1}(\tdga) \to H^{m-k-1}(\tdga)^\vee$, and set
$V^{k+1} := C^{k+1} \oplus N^{k+1}$, where $d \colon V^{k+1} \to \tdga^{k+2}$
has kernel $C^{k+1}$ and maps $N^{k+1}$ isomorphically to a subspace of
$\clalg^{k+2}$ (the closed subspace of $\tdga^{k+2}$) representing the kernel
of $\alpha \cap : H^{k+2}(\tdga) \to H^{m-k-2}(\tdga)^\vee$.

We need to study $H^m(\edga)$. Note that
$\edga^m = \tdga^m \; \oplus \; V^{k+1} \otimes \tdga^{m-k-1}$, and the closed
subspace is contained in $\tdga^m \; \oplus \; V^{k+1} \otimes \clalg^{m-k-1}$.
Therefore $H^m(\edga)$ can be written as a direct sum of
the images of $H^m(\tdga)$ and $C^{k+1} \otimes H^{m-k-1}(\tdga)$ and a direct
complement $W$. 

Note that
$\edga^{m-1} = \tdga^{m-1} \; \oplus \; V^{k+1} \otimes \tdga^{m-k-2}$.
Therefore we have $d\edga^{m-1} \; \cap \; C^{k+1} \otimes \clalg^{m-k-1} =
C^{k+1} \otimes d\tdga^{m-k-2}$, so the map
$C^{k+1} \otimes H^{m-k-1}(\tdga) \to H^m(\edga)$ is injective.
On the other hand, the kernel of $H^m(\tdga) \to H^m(\edga)$ consists of
classes represented by elements of
$\tdga^m \cap {d(V^{k+1} \otimes \tdga^{m-k-2})}
= dN^{k+1} \otimes \clalg^{m-k-2}$. Since this is contained
in the kernel of $\alpha$ by construction, $\alpha$ factors through
$\phi_\#$.

We can therefore define the restriction of $\alpha_\edga$
to the image of $H^m(\tdga)$ by requiring that
$\alpha = \phi_\#^\vee \alpha_\edga$. Further we define the restriction
to $C^{k+1} \otimes H^{m-k-1}(\tdga)$ to be the natural map arising from
$C^{k+1}$ being a subspace of $H^{m-k-1}(\tdga)^\vee$, and choose the
restriction to $W$ to be~0.
It remains to prove that this $\alpha_\edga \in H^m(\edga)^\vee$ is a
$(k{+}1)$-partial Poincar\'e class.

For $m-k \leq i \leq k$, $H^i(\edga) \cong H^i(\tdga)$, and it is easy
to see that $\alpha_\edga \cap$ is equivalent to the isomorphism
$\alpha \cap$. Meanwhile $H^{k+1}(\edga) \cong H^{k+1}(\tdga) \oplus C^{k+1}$,
and $\alpha_\edga \cap :
H^{k+1}(\edga) \to H^{m-k-1}(\edga)^\vee \cong H^{m-k-1}(\tdga)^\vee$ equals
the injective map $\alpha \cap$ on the $H^{k+1}(\tdga)$ summand, and
the inclusion $C^{k+1} \into H^{m-k-1}(\tdga)^\vee$ on $C^{k+1}$. Since we chose
$C^{k+1}$ to be a direct complement of the image of $\alpha \cap$ that means
that $\alpha_\edga \cap$ is an isomorphism on $H^{k+1}(\edga)$ too.
Finally,
$H^{k+2}(\edga) = \clalg^{k+2}/dN^{k+1} \cong H^{k+2}(\tdga)/\ker(\alpha\cap)$,
and $\alpha_\edga \cap : H^{k+2}(\tdga)/\ker(\alpha\cap) \to
H^{m-k-2}(\edga)^\vee \cong H^{m-k-2}(\tdga)^\vee$ is the map induced by
$\alpha \cap$, so injective.

\smallskip \noindent 
\ref{it:PD_unique} follows by induction from the following claim.
 
\smallskip
{\narrower\noindent
Let $\edga_1$ and $\edga_2$ be minimal Sullivan algebras generated in degree
$\leq k{+}1$, with $(k{+}1)$-partial Poincar\'e classes
$\alpha_j \in H^m(\edga_j)^\vee$.
Suppose $\tau \colon \edga_{1 \trc{k}} \to \edga_{2 \trc{k}}$ is an isomorphism
such that the class $\tau_\#^\vee \alpha_2 \in H^m(\edga_{1 \trc{k}})$ equals
the restriction of $\alpha_1$.
Then there exists an isomorphism $\phi \colon \edga_1 \to \edga_2$ such that
$\phi_\trc{k} = \tau$ and $\phi_\#^\vee \alpha_2 = \alpha_1$.\par}

\smallskip\noindent
Setting $\tdga_j := \edga_{j \trc{k}}$, we can use the argument in (ii) to
describe the generating spaces of $\edga_j$. This involves 
choices of $C^{k+1}_j \subseteq H^{m-k-1}(\tdga_j)^*$ and
$dN^{k+1}_j \subseteq \tdga_j^{k+1}$, and we choose
$C^{k+1}_1 = \tau_\#^\vee (C^{k+1}_2)$ and
$dN^{k+1}_2 = \tau(dN^{k+1}_1)$. For any linear map
$\kappa \colon N^{k+1}_1 \to \clalg^{k+1}_2$, we can define
an isomorphism $\phi_\kappa \colon \edga_1 \to \edga_2$ by setting it
to equal $\tau$ on $\edga_{1 \trc{k}}$, $\tau_\#$ on $C^{k+1}_1$,
and $\kappa + d^{-1} {\circ} \tau {\circ} d$ on $N^{k+1}_1$ (taking the
inverse of $d : N^{k+1}_2 \to dN^{k+1}_2$)---indeed any isomorphism
$\phi\colon \edga_1 \to \edga_2$ such that $\phi_\trc{k} = \tau$
is of this form.

It remains to understand $\phi_\kappa^\vee \alpha_2$. In the decomposition
from (ii) of $H^m(\edga_1)$ as the direct sum of the image of
$H^m(\edga_{1 \trc{k}})$, $C^{k+1}_1 \otimes H^{m-k-1}(\edga_{1 \trc{k}})$
and $W$, the restrictions of $\alpha_1$ and $\phi_\kappa^\vee\alpha_2$ to the
first two summands agree for any $\kappa$.

Let $\calw \subseteq \edga_1^m$ be a subspace of closed representatives of $W$.
The projection
$p \colon \edga_1^m \to N^{k+1} \otimes \tdga_1^{m-k-1}$ (with kernel
$\tdga_1^m \; \oplus \; C_1^{k+1} \otimes \tdga_1^{m-k-1}$) maps
$\calw \into N_1^{k+1} \otimes \clalg_1^{m-k-1}$.
Let us now explain that the induced map
$p \colon W \to N_1^{n+1} \otimes H^{m-k-1}(\tdga_1)$ is injective too.

Suppose that for some $w \in \calw$, $p(w) = \sum n_j \otimes dx_j$,
for $n_j \in N_1^{k+1}$ and $x_j \in \tdga_1^{m-k-2}$. Note that
$\sum dn_j \otimes x_j \in \tdga_1^{k+2} \tdga_1^{m-k-2} \subseteq \tdga_1^m$.
Therefore $w - d(\sum n_j \otimes x_j)$, which represents the same class in
$H^m(\tdga_1)$ as $w$, lies in the kernel of $p$. But $W$ is by definition a
direct complement to the space of classes represented by elements of the
kernel of $p$. Hence $W \into N_1^{n+1} \otimes H^{m-k-1}(\tdga_1)$.

The restriction of $(\phi_\kappa)_\#^\vee\alpha_2 - (\phi_0)_\#^\vee\alpha_2$
to $W$ equals the composition of $p$ with the linear map
$N_1^{k+1} \otimes H^{m-k-1}(\tdga_1) \to \Q, \; n \otimes x \mapsto
\alpha_2(\kappa(n)\tau_\# x)$. Because
$\tau_\# : H^{m-k-1}(\tdga_1) \to H^{m-k-1}(\tdga_2)$ and
$\alpha_2 \cap \colon H^{m-k-1}(\tdga_2) \to H^{k+1}(\tdga_2)^\vee$ are
isomorphisms, any linear functional on $N_1^{k+1} \otimes H^{m-k-1}(\tdga_1)$
is realised this way for some $\kappa \colon N^{k+1}_1 \to \clalg^{k+1}_2$.
Thus by adjusting the choice of $\kappa$, the restriction of
$(\phi_\kappa)_\#^\vee\alpha_2$ to $W$ can be made to equal $\alpha_1$.
\end{proof}

\subsection{Partial Poincar\'e classes and the \bmp}

We pointed out above that in Proposition \ref{prop:PD}\ref{it:PD_exist} the
restriction of $\alpha$ to $H^m_\trc{k+1}(\tdga)$ is responsible
for reconstructing the cohomology algebra of $\mnalg$.
If we now require that $\tdga$ is \tkc{} (in addition to being minimal and
generated in degree $\leq k$), we consider to what extent the remaining
components of $\alpha$ are determined by the \bmp.

In the setting of primary interest in this paper, \ie when $m \leq 5n{-}3$,
and for suitable $k$ we can show that essentially all of those remaining
components are captured by the \bmp.
For $m = 5n{-}2$ we can at least show that the components of the \bmp{}
are independent of the components responsible for reconstructing the
Poincar\'e cohomology algebra.

In the statement of the next lemma, we use the following notation:
given $k$ and $\tdga$, let $\bmspak{\tdga}$ be the intersection
of $\bmspa{\tdga}$ with the image in $\gsym^2\gsym^2 H^*(\tdga)$ of
$((H^*)^{\otimes 3})^{\leq k{+}2} \otimes H^*$. In other words, $\bmspak{\tdga}$
is the part of $\bmspa{\tdga}$ that involves only quadruples of classes where
the sums of the three lowest degrees is at most $k{+}2$.
Note that if $\tdga$ is \tkc{} then certainly $\bmspa{\tdga}$ involves only
classes of degree $\geq n$, so if $k \leq 3n{-}3$ then $\bmspak{\tdga} = 0$.
The proof of \ref{it:trunc_exist} is tidier in this case, but to prove the
realisation result Theorem \ref{thm:minalg}\ref{it:minexist} when $m = 5n{-}2$
we need to allow $k = 3n{-}2$. 

\begin{lem}
\label{lem:trunc_bmp}
Let $n \geq 2$ and $m \geq 4n{-}1$.
\begin{enumerate}
\item \label{it:trunc_exist}
Let $k \leq 4n{-}3$, and let $\tdga$ be an \tkc{} minimal Sullivan algebra
generated in degree $\leq k$.
\begin{enumerate}
\item If $m \leq 6n{-}3$ then the intersection of $H_\trc{k+1}^m(\tdga)$
(\ie the subspace of $H^m(\tdga)$ generated by products
of classes of degree $\leq k{+}1$) with
the image of $\princ \colon \bmspam{\tdga} \to H^m(\tdga)$
is contained in $\princ(\bmspakm{\tdga})$.
\item If $m \leq 5n{-}2$ then the kernel of
$\princ\colon \bmspam{\tdga} \to H^m(\tdga)$ is contained in $\bmspakm{\tdga}$.
\end{enumerate}

\item \label{it:trunc_unique}
Suppose $m \leq 5n{-}3$ and $k \leq 3n{-}3$.
Let $\tdga_1$ and $\tdga_2$ be \tkc{} minimal Sullivan algebras
generated in degree $\leq k$.
Given an isomorphism $G \colon H^{\leq k}(\tdga_1) \to H^{\leq k}(\tdga_2)$
of the truncated cohomology rings, and $k$-partial Poincar\'e
classes $\alpha_j \in H^m(\tdga_j)^\vee$,
there is an isomorphism $\tau \colon \tdga_1 \to \tdga_2$ such that
$\tau_\# = G$ on $H^{\leq k}(\tdga_1)$ and $\tau_\#^\vee \alpha_2 = \alpha_1$
if and only if the diagram below commutes.
\[ \xymatrix{
(\Lambda H^{\leq k}(\tdga_1))^m \oplus \bmspam{\tdga_1} \ar[d]^G \ar[r]
& H^m(\tdga_1) \ar[r]^(0.65){\alpha_1} & \Q \ar[d] \\
(\Lambda H^{\leq k}(\tdga_2))^m \oplus \bmspam{\tdga_2} \ar[r]
& H^m(\tdga_2) \ar[r]^(0.65){\alpha_2} & \Q } \]
%
\end{enumerate}
\end{lem}

\begin{rmk}
If $\tdga$ is generated in degree $\leq k$, then
$\Lambda H^{\leq j}(\tdga) \into H^*(\tdga)$ for any $j > k$.
If $\tdga$ is \tkc, then the image $H_\trc{j}^*(\tdga)$ is the same for
all $k \leq j < 3n{-}2$; however it is \emph{not} necessarily isomorphic to
$\Lambda H^{\leq k}(\tdga)$. In \ref{it:trunc_unique}, $G : H^{\leq k}(\tdga_1) \to H^{\leq k}(\tdga_2)$ would therefore not automatically induce a map
$H_\trc{k}^*(\tdga_1) \to H_\trc{k}^*(\tdga_2)$ (never mind extend to
$H^*(\tdga_1) \to H^*(\tdga_2)$) if we did not also assume the commutativity
of the diagram.

Since $\bmspak{\tdga} = 0$ if $k \leq 3n{-}3$, \ref{it:trunc_exist} implies
that the maps $\princ\colon \bmspam{\tdga_i} \to H^m(\tdga_i)$ in
\ref{it:trunc_unique} are injective, with image transverse to
$H_\trc{k}^m(\tdga_i) = H_\trc{k+1}^m(\tdga_i)$.

Insisting that $k \geq m{-}2n$ in \ref{it:trunc_unique} ensures that if
$\mnalg$ is \tkc{} then $\bmspam{\mnalg_\trc{k}} = \bmspam{\mnalg}$.
The statement 
would not in general be true if we instead set $k := [m/2]$.
\end{rmk}

\begin{proof} The argument is similar to the induction steps in the proof
of Proposition \ref{prop:PD}.

\smallskip

\noindent
\ref{it:trunc_exist}
We begin by describing a generating space $V$ for $\tdga = \Lambda V$,
as before writing $V^i$ as a direct sum of the closed subspace $C^i$
and a direct complement $N^i$.

By a trivial recursion we find that
$V^i = 0$ for $0 < i < n$, and
$V^i = C^i$ for $n \leq i \leq 2n{-}2$ (\ie $V^i$ consists of only closed
elements). Further $C^i \cong H^i(\tdga)$ for $i \leq 2n{-}1$.

For $2n \leq i \leq 3n{-}2$, the closed subspace of $\tdga_\trc{i-1}^i$ is
precisely $(\calg^2 C^*)^i$. 
Thus $C^i$ is identified with a direct complement to the image of
$(\calg^2 C^*)^i \to H^i(\tdga)$, while $d$ maps $N^{i-1}$ isomorphically to
its kernel.

For $3n{-}1 \leq i \leq k{+}1$, the closed subspace of $\tdga_\trc{i-1}^i$ is a
direct sum of $(\calg^2 C^* \oplus \calg^3 C^*)^i$ and the
closed subspace of $(C^* \otimes N^*)^i$; we let $T^i$ denote the latter term.
Thus $C^i$ is identified with a direct complement to the image of
$(\calg^2 C^* \oplus \calg^3 C^* \oplus T)^i \to H^i(\tdga)$ (for $i \leq k$),
and $N^{i-1}$ is mapped isomorphically to the kernel.

Let $\wt \ccalg := \Lambda C \subseteq \tdga$.
Then we can decompose $\tdga^*$ as a direct sum of
$\wt \ccalg^*$, $\wt \ccalg^* \otimes N^*$, $\wt \ccalg^* \otimes \gsym^2N^*$
etc.
The first term consists of only closed forms, while the closed
subspace of the second term contains $C^* \otimes T^*$.

\medskip

\noindent
(a) 
$d(\wt \ccalg \otimes N^*) \subseteq \wt \ccalg \; \oplus \; \wt\ccalg \otimes T^*$,
while any element of $\wt \ccalg \; \oplus \; \wt \ccalg \otimes N^*$ that is
the differential of something in $\wt \ccalg \otimes \gsym^{\geq 2}N^*$
belongs to $\wt \ccalg \otimes N^*$.

Note that for $i \leq 3n{-}1$, 
$E^i \cong \ker \left((\calg^2 C^*)^i \to H^{i+1}(\tdga)\right)$,
which is the injective image under $d$ of a subspace
$\overline N^{i-1} \subseteq N^{i-1}$
(in fact equality holds for $i \leq 3n{-}2$).
Because $\bmspam{\tdga}$ only involves $E^i$ for $i \leq m+1 - 2n \leq 3n-1$,
we can therefore choose the map $\gamma : E^i \to \tdga^{i-1}$ in the
definition of $\princ$ to take values in $\overline N^{i-1}$,
so that $d\gamma$ takes values in $\gsym^2 C^* \subseteq \wt \ccalg^*$.
In particular, the image of $\princ : \bmspam{\tdga} \to H^m(\tdga)$ is
represented by elements of $N^* \otimes \wt \ccalg^*$.

On the other hand, $H_\trc{k+1}^m(\tdga)$ is the subspace represented by the
degree $m$ part of $\wt \ccalg^* \; \oplus \; T^* \otimes \wt \ccalg^*$.
Therefore the intersection of the images in $H^m(\tdga)$
can be represented by elements in the intersection
$\wt \ccalg^* \otimes T^*$, corresponding to $\bmspakm{\tdga}$.

\medskip

\noindent
(b) Using once more that
$\kc^i \cong \ker \left((\calg^2 C^*)^i \to H^{i+1}(\tdga)\right)
= d\overline N^{i-1}$ in the degrees that contribute to the degree $m+1$ part of
$\kc^* \otimes \kc^*$, the degree $m+1$ part of $\kersym{\kc^* \otimes \kc^*}$
is isomorphic to the closed subspace $\calk^m$ of
the degree $m$ part of $\overline N^* \otimes d\overline N^*$.
Further $\agrad^2 \kc^*$ is mapped to $d(\overline N^* \otimes \overline N^*)$.
We therefore obtain an isomorphism
$\bmspam{\tdga} \to (\kersym{\kc^* \otimes \kc^*}/\agrad^2 \kc^*)^{m+1}
\to (\calk^*/d(\overline N^* \otimes \overline N^*))^m$
(\cf Remark \ref{rmk:exactpart}).

We see that $\princ : \bmspam{\tdga} \to H^m(\tdga)$ is the composition of
that isomorphism with the projection to $H^m(\tdga)$.
The kernel of the latter map is represented by elements of
$d(\wt \ccalg \otimes N^{\geq 3n-2})$, whose preimage in $\bmspam{\tdga}$
is contained in $\bmspakm{\tdga}$.
\medskip

\noindent
\ref{it:trunc_unique}
To describe the homomorphism $\tau : \tdga_1 \to \tdga_2$, we need to specify
its values on the generating spaces $V^i_1 \subset \tdga_1$.
The values on $C_1^i \subseteq V_1^i$ are determined by $G$.
Thus the only flexibility that remains for adjusting $\tau_\#^\vee \alpha_2$
is the restriction of $\tau$ to $N_1^*$.

We now claim that any closed element of $\tdga^m$ can be written as a sum of
products of generators such that at most one factor in each term is not closed.
The hypothesis that $m \leq 5n{-}3$ ensures that we can decompose
$\tdga^m$ as a direct sum of the degree $m$ parts of
$\wt \ccalg^m$, $N^* \otimes \wt \ccalg^*$ and $\gsym^2 N^*$
(in particular we do not need $\wt \ccalg^* \otimes \gsym^2 N^*$).
If $x \in \tdga^m$ has component
$\sum m_j \otimes n_j$ in $N^i \otimes N^{m-i}$, then the
$N^i \otimes \wt \ccalg^{m-i+1}$ component of $dx$ is $\sum m_j \otimes dn_j$.
That vanishes only if $\sum m_j \otimes n_j$ does, proving the claim.

For $2n -2 \leq j \leq k$, let $L_j$ be the closed subspace of
\[
N^{2n-1} \otimes \clalg^{m-2n+1} \; \oplus \cdots \oplus \;
N^j \otimes \clalg^{m-j} \; \oplus \; N^{j+1} \otimes dN^{m-j-2} \; \oplus
\cdots \oplus \; N^k \otimes dN^{m-k-1} , \]
where as usual $\clalg^* \subseteq \tdga^*$ is the closed subalgebra.
So
\[ \calk^m = L_{2n-2} \subseteq L_{2n-1} \subseteq \cdots \subseteq L_k,\]
and $\wt \ccalg^m \oplus L_k = \clalg^m$.
Since the hypothesis means that $\tau_\#^\vee \alpha_2$
and $\alpha_1$ agree on $\wt \ccalg^m \oplus \calk$ (the first term represents
elements of $H_\trc{k{+}1}^m(\tdga)$, and the second
the image of $\princ$), the desired conclusion
follows by induction from the following claim.

\smallskip
{\narrower\noindent
Let $2n{-}1 \leq j \leq k$, and let
$\tau : \tdga_1 \to \tdga_2$ be an isomorphism.
Then there exists another isomorphism $\phi$ with
$\phi_\trc{j-1} = \tau_\trc{j-1}$
such that $\phi_\#^* \alpha_2$ and $\alpha_1$ agree on
$\wt \ccalg^m \oplus L_j$ if and only if $\tau_\#^\vee \alpha_2$ and $\alpha_1$
agree on $\wt \ccalg^m \oplus L_{j-1}$.\par}

\smallskip\noindent
Set $\phi_\kappa := \tau + \kappa$, for some linear map
$\kappa : N^j \to \clalg^j$. Let $W$ be a direct complement
to $L_{j-1}$ in~$L_j$. Then $p \colon W \into N^j \otimes H^{m-j}(\tdga)$.
Now $(\phi_\kappa)_\#^\vee \alpha_2 = \tau_\#^\vee\alpha_2$ on $L_{j-1}$,
while on $W$ the difference equals the composition of $p$ with the linear map
$N_1^j \otimes H^{m-j}(\tdga_1) \to \Q, \; n \otimes x \mapsto
\alpha_2(\kappa(n)\tau_\# x)$. Because
$\tau_\# : H^{m-j}(\tdga_1) \to H^{m-j}(\tdga_2)$ and
$\alpha_2 \cap \colon H^{m-j}(\tdga_2) \to H^j(\tdga_2)^\vee$ are
isomorphisms, any linear functional on $N_1^j \otimes H^{m-j}(\tdga_1)$
is realised this way for some $\kappa \colon N^j_1 \to \clalg^j_2$.
Thus by adjusting the choice of $\kappa$, the restriction of
$(\phi_\kappa)_\#^\vee\alpha_2$ to $W$ can be made to equal $\alpha_1$.
\end{proof}


\begin{ex}
\label{ex:obstr_detail}
Suppose $H^*$ is an algebra with $x_1$, $x_2$, $x_3 \in H^n$
such that the products $x_1^2$, $x_2^2$, and $x_1x_2$ vanish in $H^{2n}$,
and $y_1 = x_1x_3$ and $y_2 = x_2x_3 \in H^{2n}$ are linearly
independent.
Then $b := (x_1y_1)x_2^2 - (y_1x_2)(x_2x_1) - (x_2y_2)x_1^2 + (x_1y_2)(x_1x_2)$
is a non-zero element of $\bmsp^{5n}_{3n-2}(H^*)$.

If a DGA $\tdga$ has $H^*(\tdga)$ of this form, then $\princ(b) = 0$.
For in the definition of $\princ$, choose the linear map
$\alpha : H^*(\tdga) \to \tdga^*$ such that
$\alpha(y_1) = \alpha(x_1)\alpha(x_3)$ and
$\alpha(y_2) = \alpha(x_2)\alpha(x_3)$.  We can then choose
$\gamma : \kc^* \to \tdga^{*-1}$ such that
$\gamma(x_1y_1) = \gamma(x_1^2)\alpha(x_3)$
and $\gamma(x_2y_1) = \gamma(x_1x_2)\alpha(x_3)$.
Then $\princ(b) \in H^{5n-1}(\tdga)$ is represented by
\[ \gamma(x_1y_1)\alpha(x_2)^2 - \gamma(y_1x_2)\alpha(x_2)\alpha(x_1)
- \gamma(x_1^2)\alpha(x_2)\alpha(y_2) + \alpha(x_1)\alpha(y_2)\gamma(x_1x_2)
= 0. \]
There exist \tkc{} $(5n{-}1)$-dimensional Poincar\'e algebras $H^*$ with the
above property, as can be seen for instance by applying
Lemma \ref{lem:partial}\ref{it:reconstr} to an algebra supported in degree $0$,
$n$ and~$2n$, with $\alpha = 0$ as a $(3n{-}2)$-partial Poincar\'e class of
degree $5n{-}1$. Thus the realisation statement of Theorem
\ref{thm:minalg}\ref{it:minexist} does not extend beyond $m \leq 5n{-}2$.

Similarly, there exist \tkc{} minimal DGAs $\tdga$ generated in
degree $\leq 3n{-}1$ such that $H^*(\tdga)$ has the above property.
This demonstrates that if $m \geq 5n{-}1$, then for all $k \geq m{-}2n$
there can be a minimal DGA $\tdga$ generated in degree $\leq k$ such that
$\princ$ fails to be injective on $\bmspam{\tdga}$.
That is where our proof of Theorem \ref{thm:minalg}\ref{it:minexist}  
breaks down for $m \geq 5n{-}1$.
\end{ex}

From Lemma \ref{lem:trunc_bmp} we can now deduce a result that links up with
Proposition \ref{prop:PD}.

\begin{prop}
\label{prop:trunc}
Let $4n{-}1 \leq m \leq 5n{-}2$ and $k = m{-}2n$.
Then given an \tkc{} $m$-dimensional Poincar\'e duality algebra
$(\bar H^*, \bar \alpha)$
and a linear map $F : \bmspam{\bar H^*} \to \Q$, there exists
a minimal Sullivan algebra $\tdga$ generated in degree $\leq k$, with
\begin{itemize}
\item an injection $\phi : H^{\leq k+1}(\tdga) \into \bar H^{\leq k+1}$
that is an isomorphism in degree $\leq k$ and
\item a $k$-partial Poincar\'e duality class
$\alpha \in H^m(\tdga)$
\end{itemize}
such that
\begin{itemize}
\item the restriction of $\alpha$ to $H_\trc{k{+}1}^m(\tdga)$
is $\phi^\vee \bar \alpha$ (note that $\phi$ induces a well-defined map
$H_\trc{k{+}1}^*(\tdga) \to \bar H^*$,
because $\tdga$ is generated in degree $\leq k$), and
\item the \bmp{} $\princ : \bmspam{\tdga} \to H^m(\tdga)$ satisfies
$\alpha \circ \princ = F \circ \phi'$, 
where the map $\phi' : \bmspam{\tdga} \to \bmspam{\bar H}$ is the isomorphism induced by $\phi$.
\end{itemize}
Moreover, if $m \leq 5n{-}3$ then $\tdga$ is unique.
\end{prop}

\begin{proof}
If $m \leq 5n{-}3$, so that $k \leq 3n{-}3$, then it is easy construct
the unique minimal $\tdga$ generated in degree $\leq k$
with an injection $\phi : H^{\leq k+1}(\tdga) \into \bar H^{\leq k+1}$
that is an isomorphism in degree $\leq k$.
Then Lemma \ref{lem:trunc_bmp}\ref{it:trunc_exist} shows that
$\princ : \bmspam{\tdga} \to H^m(\tdga)$ is injective, with image
transverse to $H_\trc{k+1}^m(\tdga)$.  
Therefore we can define $\alpha \in H^m(\tdga)^\vee$ by setting
$\alpha = \bar \alpha \circ \phi$ on $H_\trc{k+1}^m(\tdga)$, and
$\alpha = F \circ \princ^{-1}$ on the image of $\princ$,
and extending arbitrarily to the rest of $H^m(\tdga)$.
That $\alpha$ is a $k$-partial Poincar\'e class follows from
Lemma \ref{lem:partial}\ref{it:restrict}.
The relevant uniqueness follows immediately from
Lemma \ref{lem:trunc_bmp}\ref{it:trunc_unique}.

It remains to prove the existence claim in the borderline case $m = 5n{-}2$,
and this requires a little more work.
While the generating spaces $V^i$ of $\tdga$ are determined by $\bar H^*$ in
degree $i < k = 3n{-}2$ (as in the proof of
Lemma \ref{lem:trunc_bmp}\ref{it:trunc_exist}), we must also describe
$V^{3n{-}2}$. $C^{3n{-}2}$ is just a direct complement to the image of
$(\gsym^2 C^*)^{3n-2} \to H^{3n{-}2}$ like before.

The closed subspace of $\tdga_\trc{3n{-}2}^{3n-1}$ is
$(\gsym^2 C^*)^{3n{-}1} \oplus T^{3n{-}1}$, where $T^{3n{-}1}$ is the closed
subspace of $C^n \otimes N^{2n{-}1}$. Then
$T^{3n{-}1} \cong \kersym{H^n \otimes E^{2n}}$, and
\[ C^{2n-1} \otimes T^{3n{-}1} \cong
H^{2n-1} \otimes \kersym{H^n \otimes E^{2n}}
\supseteq H^{2n-1} \otimes (H^n)^{\otimes 3} \cap
\kersym{E^{3n-1} \otimes E^{2n}} . \]
The RHS maps surjectively onto
$\bmspakd{5n-1}{H^*} \stackrel{\phi'}{\cong} \bmspakd{5n-1}{\bar H^*}$.
We can this restrict the given $F$ to  $\bmspakd{5n-1}{\bar H^*} \to \Q$,
pull back to
$H^{2n-1} \otimes (H^n)^{\otimes 3} \; \cap \; \kersym{E^{3n-1} \otimes E^{2n}}$
and then
extend arbitrarily to a map $\wt F : C^{2n-1} \otimes T^{3n-1} \to \Q$.
That is equivalent to an
$F' : T^{3n-1} \to (C^{2n-1})^\vee \cong \bar H^{3n{-}1}$,
characterised by $\wt F(c \otimes t) = \bar \alpha(\phi(c)F'(t))$.
On the other hand, the algebra induces a map
\mbox{$a : \gsym^2 C^* \to \bar H^*$}.
We set $d : N^{3n-2} \to \tdga_\trc{3n-2}^{3n-1}$ to map isomorphically to
the kernel of
\[ a + F' : (\gsym^2 C^*)^{3n{-}1} \oplus T^{3n{-}1} \to
\bar H^{3n{-}1} . \]

Having thus defined $\tdga$, 
set $\phi : H^{\leq 3n{-}1}(\tdga) \to \bar H^{\leq 3n{-}1}$ to be
induced by $a$ in degrees $\leq 3n{-}2$, and by $a + F'$ in degree $3n{-}1$.

Note that $\princ(\bmspakd{5n-1}{\tdga})$ can be represented by elements
of $C^{2n-1} \otimes T^{3n-1}$.
The choices of $d : N^{3n-2} \to \tdga_\trc{3n-2}^{3n-1}$ and $\phi$ ensure
that the restriction of $\bar \alpha \circ \phi$ to
$\princ(\bmspakd{5n-1}{\tdga})$ equals the restriction of $\wt F$, which
in turn equals $F \circ \princ^{-1}$ by construction. 
By Lemma \ref{lem:trunc_bmp}\ref{it:trunc_exist} the intersection
of $H_\trc{k+1}^m(\tdga)$ and the image of $\princ$
is contained in $\princ(\bmspakd{5n-1}{\tdga})$. 
We can therefore once more define $\alpha \in H^m(\tdga)^\vee$ by
setting $\alpha = \bar \alpha \circ \phi$ on $H_\trc{k+1}^m(\tdga)$, and
$\alpha = F \circ \princ^{-1}$ on the image of $\princ$ (and extending
arbitrarily).
\end{proof}

Combining Proposition \ref{prop:trunc} and Proposition \ref{prop:PD} completes
the proof of Theorem \ref{thm:minalg}.

\subsection{Minimal models and manifolds}
\label{subsec:models}

A \emph{minimal model} of a DGA $\dga$
is a minimal Sullivan algebra $\mnalg$ together
with a quasi-isomorphism $\qi : \mnalg \to \dga$,
\ie a DGA
homomorphism whose induced map $\qi_\# : H^*(\mnalg) \to H^*(\dga)$ is an
isomorphism.

Recall that a minimal model of a CW complex $X$ is a minimal model of
$\drpl(X)$ and that for simply-connected CW complexes with rational cohomology
of finite type every quasi-isomorphism of minimal models is realised by a
rational homotopy equivalence of spaces, see Sullivan \cite[\S 8]{sullivan77}
and F\'elix-Halperin \cite[Proposition 17.13]{felix01}.
Theorem \ref{thm:type} now follows directly from the following claim. 

\begin{cor}
\label{cor:minmodel}
Let $\dga$ and $\dga'$ be \tkc{} 
Poincar\'e DGAs of dimension $m \leq 5n{-}3$ ($n \geq 2$), with minimal models
$\qi : \mnalg \to \dga$ and $\qi' : \mnalg' \to \dga'$.
If $G : H^*(\dga) \to H^*(\dga')$ is an isomorphism of the cohomology rings
then there exists a DGA isomorphism
$\phi : \mnalg \to \mnalg'$ such that
$\qi_\# \circ \phi_\# = G \circ \qi'_\#$
if and only if the diagram below commutes.
\[ \xymatrix{
\bmspam{\dga} \ar[d]^{\princ} \ar[r]^{G}
& \bmspam{\dga'} \ar[d]^{\princ'} \\
H^m(\dga) \ar[r]^{G} & H^m(\dga')}  \]
\end{cor}

\begin{proof}
Apply Theorem \ref{thm:minalg}\ref{it:minunique}
to $(q'_\#)^{-1} \circ G \circ q_\#$.
\end{proof}

Next recall that a DGA $\dga$ is said to be \emph{formal} if there is a
quasi-isomorphism $\wh \qi : \mnalg \to (H^*(\dga), 0)$ from its minimal
model $\mnalg$---in other words, if $\dga$ and $(H^*(\dga), 0)$ have the same
minimal model.
A space $X$ is formal if its DGA of piecewise linear de Rham forms is.
The following proposition is therefore the algebraic formulation of Theorem
\ref{thm:obstruction}.

\begin{cor}
\label{cor:obstruction}
An \tkc{} 
Poincar\'e DGA $\dga$ of dimension $m \leq 5n{-}3$ is formal if and
only if the \bmp{} $\princ \colon \bmspam{\dga} \to H^m(\dga)$ is trivial.
\end{cor}

\begin{proof}
Since the DGA $(H^*(\dga), 0)$ has $\princ = 0$, the functoriality of the
\bmp{} implies that if $\dga$ is formal then its minimal model, and hence
$\dga$ itself, also have $\princ = 0$.

Conversely if $\princ = 0$ then we can let $\dga' := (H^*(\dga),0)$,
set $G \colon H^*(\dga) \to H^*(\dga')$ to be the tautological isomorphism and
apply Corollary \ref{cor:minmodel} to deduce that $\dga$ is formal.
\end{proof}

\begin{rmk}
\label{rmk:3formality}
Our reasoning has been guided by the notion of $k$-formality of
Fern\'andez and Muñoz~\cite{fernandez05}. They define $M$ to be
\emph{$k$-formal} if one can choose the generating set $V$ for a minimal
model $\mnalg$ and direct complements
$N^i$ to $C^i \subseteq V^i$ for $i \leq k$ so
that the cohomology of the ideal
$I_k := N^{\leq k} \mnalg_\trc{k} \subseteq \mnalg_\trc{k}$
maps trivially to $H^*(M)$.
According to \cite[Theorem 3.1]{fernandez05}, a
closed orientable manifold of dimension $\leq 2k{+}1$ is formal if and only if
it is $k$-formal---this can also be deduced from Proposition \ref{prop:PD}
(under the simplifying assumption of simple-connectedness).

For $m \leq 5n{-}3$ and $m{-}2n \leq k \leq 3n{-}3$, the algebraic
considerations in Lemma \ref{lem:trunc_bmp} essentially
identify the \bmp{} as a complete obstruction to $k$-formality for
closed \tkc{} $m$-manifolds. One could thus prove Theorem \ref{thm:obstruction}
more briefly by appealing to the results of Fern\'andez and Muñoz, but we have
set up the argument to prove the more general Theorem \ref{thm:type} too.
\end{rmk}

We now move from rational classification to the classification of
manifolds up to finite ambiguity. Specialising the results of
Kreck and Triantafillou \cite{kreck91a} to the context of this paper, we can
interpret their results in terms of the \bmp, and deduce the
following superficially stronger version of
Theorem \ref{thm:classification_up_to_finite}.

\begin{prop}
\label{prop:classification_up_to_finite}
For $n \geq 2$ and $5 \leq m \leq 5n{-}3$, closed \tkc{} $m$-manifolds $M$ are
classified up to finite ambiguity by the truncated integral cohomology ring
$H^{\leq \frac{m}{2}+1}(M;\Z)$, the cohomology algebra $H^*(M)$, rational
Pontrjagin classes $p_k(M) \in H^{4k}(M)$ and
the \bmp{} $\bmspam{M} \to H^m(M)$.
\end{prop}

\begin{proof}
Kreck and Triantafillou \cite[Theorem 2.2]{kreck91a} prove that the
diffeomorphism type of a closed simply connected $M$ of dimension $m \geq 5$
with formal $([\frac{m}{2}] {+} 1)$-skeleton is determined up to finite
ambiguity by the truncated cohomology ring $H^{\leq \frac{m}{2} + 1}(M;\Z)$,
the rational Pontrjagin classes,
and $\alpha_M \in H^m(\mnalg_\trc{\frac{m}{2}})^\vee$; here
$\mnalg_\trc{\frac{m}{2}}$ is the subalgebra of the minimal model of $M$
generated by elements of degree $\leq \frac{m}{2}$,
and $\alpha_M$ is the pull-back of $\int_M \in H^m(M)^\vee$
under $H^m(\mnalg_\trc{\frac{m}{2}}) \to H^m(M)$.
If $M$ is \tkc{} and $m \leq 5n{-}3$ then the $([\frac{m}{2}]{+}1)$-skeleton is
certainly formal, and Proposition \ref{prop:trunc} implies that $\alpha_M$ is
determined up to isomorphism by the cohomology algebra and the \bmp.
\end{proof}

\subsection{Rational realisation}
\label{subsec:realise}
In this subsection we prove Theorem \ref{thm:realisation} using
Theorem \ref{thm:minalg}\ref{it:minexist} and rational surgery, adapted to
the setting of \tkc{} manifolds.
When the dimension $m$ is not divisible by $4$, we can proceed
by making some minor adjustments to
Sullivan's proof of \cite[Theorem 13.2]{sullivan77}.
When $m=4k$ the most convenient statement of rational surgery for generalisation
to \tkc{} manifolds  is Barge's  \cite[Theorem 1]{barge76} and the best proof for these 
purposes is found in the PhD thesis of Su \cite{su09}.

We are given $(H^*, p_*, \princ)$, an $(n{-}1)$-connected $m$-dimensional
rational Poincar\'e duality algebra $H^*$ ($m \leq 5n{-}2$),
together with candidate Pontrjagin classes
$p_* \in H^{4*}$ and a linear map $\princ \colon \bmspam{H^*} \to H^{m}$
that is a candidate for the \bmp.
By Theorem \ref{thm:minalg}\ref{it:minexist}, there exists a Sullivan minimal
algebra $\mnalg$ with $H^*(\mnalg) = H^*$ and \bmp{} $\princ$.
By \mbox{\cite[\S 8]{sullivan77}} (see also \cite[Theorem 17.10]{felix01}), 
$\mnalg$ is realised by a rational space $X$
which is $(n{-}1)$-connected since $\mnalg$ is $(n{-}1)$-connected.
The cohomology classes $p_* \in H^{4*}(\mnalg) = H^{4*}(X)$ define
a map $p \colon X \to \Pi_{4i \geq n} K(\Q, 4i)$ to the indicated product of rational Eilenberg-MacLane spaces.
If $BO\an{n}$ denotes the $(n{-}1)$-connected cover of $BO$,
then the universal Pontrjagin classes on $BO\an{n}$ 
define a rational equivalence
$p\an{n} \colon BO\an{n} \to \Pi_{4i \geq n} K(\Q, 4i)$
and we let $Y$ be the 
space in the following pullback square:
\[ \xymatrix{Y \ar[d] \ar[r] & BO\an{n} \ar[d]^{p\an{n}} \ar[d] \\
X \ar[r]^(0.275){p} & \Pi_{4i \geq n}K(\Q, 4i)  }  \]

We note that $Y \to X$ is a rational equivalence since $p\an{n}$ is a 
rational equivalence.
If we set $T$ to be the Thom space of the stable bundle over $Y$
induced by the map $Y \to BO\an{n}$, then the stable homotopy groups
of $T$ satisfy 
$\pi_m^s(T)\otimes \Q \cong H_m(Y) \cong H_m(X) = \Q$,
since $Y \to X$ is a rational equivalence.
We wish to find a closed smooth $m$-manifold $M$ together with a bundle map,
\[ 
\xymatrix{ 
\nu_M \ar[d] \ar[r]^{\ol f} &
\xi \ar[d] \\
M \ar[r]^f &
Y,
}
\]
where $\nu_M$ denotes the stable normal bundle of $M$,
$\xi$ is the stable bundle over $Y$ induced from the map $Y \to BO\an{n}$
and $f \colon M \to Y$ is of non-zero degree.  
The existence of $(M, f, \ol f)$ is automatic when
$m \neq 4k$ and when $m = 4k$, it is proven by Su
\cite[Lemma 3.2.2]{su09} for the case when $BO\an{n} = BSO$.
Specifically, Su shows that there is class in $x \in \pi_{4k}^s(T)$ such that normal maps
$(M, f, \ol f)$ obtained from $x$ via the Pontrjagin-Thom isomorphism
satisfy $f_*([M]) = \alpha$.
The argument there works just as well in our case as we have the correctly adapted
assumption $(\mathrm{iii})$ that the Pontrjagin numbers defined by $\alpha$ and $p_*$ are
those of a \tkc{} manifold.\footnote{We thank Jim Davis and Zhixu Su for
explaining this point to us.}
%
%
Hence we have the desired normal map $(M, f, \ol f)$.
Since $Y$ is $(n{-}1)$-connected,
we may perform surgery below the middle dimension \cite[\S 1]{wall99} on 
$f \colon M \to Y$ to make $M$ $(n{-}1)$-connected.
We continue further with rational surgery as in the proof of \cite[Theorem 13.2]{sullivan77} to achieve that 
$f \colon M \to Y$ is a rational equivalence 
with $M$ still $(n{-}1)$-connected using the assumption when $m = 4k$, that the intersection form
defined by $\alpha$ is equivalent to a sum or squares.
Then $(H^*(M), p_*(M), \princ(M)) = (H^*(X), p_*, \princ)$,
proving Theorem \ref{thm:realisation}.

\section{Coboundaries and integrality}
\label{sec:coboundaries}

In this section we explain how to compute the \bmp{} of a closed \mbox{\tkc{}}
$m$-manifold $M$ if $M$ has a coboundary $W$ such that the restriction
homomorphism $H^*(W) \to H^*(M)$ is surjective in degree $\leq m{-}3n{+}1$.
We can use this to construct compact manifolds $W$ whose boundaries $M$ realise
a given cup-product structure and \bmp.

\subsection{Computing the \bmp{} via a coboundary}

Let $W$ be a compact \mbox{$(m{+}1)$-manifold} with boundary $M$.
We call $W$ a \emph{coboundary of $M$ over $H^{\leq s}$} if
the restriction map $j : H^*(W) \to H^*(M)$ is surjective in degrees $\leq s$.
Then we can pick a right inverse $r : H^{\leq s}(M) \to H^{\leq s}(W)$ of $j$.
We will denote $r$ by $x \mapsto \hat x$.

In what follows we will assume that $M$ is \tkc, and that $W$ is a coboundary
over $H^{\leq s}$ for $s := m{-}3n{+}1$.
Then $E^{\leq n+s}$ is contained in $\gsym^2 H^{\leq s}$.
Therefore composing $r$ with the cup product of $W$ induces a map
$E^{\leq n+s} \to H^{\leq n+s}(W), \; e \mapsto \hat e$. In fact, this takes
values in $H^*_0(W) \subseteq H^*(W)$, the image of 
the natural map $H^*(W,M) \to H^*(W)$.

The condition that $M$ is \tkc{} ensures in turn that
$\gsym^2 E^{\leq n+s}$ and $\gsym^2 E^*$ are equal in degree $\leq m{+}1$,
so in particular we obtain an induced map
$(\gsym^2 E)^{m+1} \to (\gsym^2 H^*_0(W))^{m+1}$.
The intersection form of $W$ is a well-defined map
$\lambda_W : (\gsym^2H^*_0(W))^{m+1} \to \Q$.
Let
\begin{equation}
\label{eq:aw}
A_W : (\gsym^2 E)^{m+1} \to \Q
\end{equation}
be the composition.
On the other hand, we identify the top degree part $\bmspam{M} \to H^m(M)$
of the \bmp{} with a linear map $\princ_M : \bmspam{M} \to \Q$ by composition
with the integration map of $M$.

\begin{lem}
\label{lem:coboundary}
Let $M$ be a closed \tkc{} $m$-manifold, and let $W$ be a coboundary of $M$
over $H^{\leq s}$ for $s := m{-}3n{+}1$.
Then the restriction of $A_W$ to $\bmspam{M}$ equals $\princ_M$.
\end{lem}

\begin{proof}
Choose $\wh \rp : H^{\leq s}(M) \to \drpl^{\leq s}(W)$ so that
\[ \xymatrix{
H^{\leq s}(M) \ar@{^(->}[r]^{r} \ar@{_(->}[d]_{\rp} \ar[rd]^{\wh \rp} & H^{\leq s}(W) \\
\drpl^{\leq s}(M) & \drpl^{\leq s}(W) \ar@{->>}[l] \ar@{->>}[u] } . \]
commutes. For $e e' \in (\gsym^2E^{\leq s+ n})^{m+1}$, we then have that
$\wh \rp^2(\ke), \wh \rp^2(\ke') \in \drpl^*(W)$ are
representatives of $\hat \ke$ and $\hat \ke' \in H^*_0(W)$.
If $d\prp(\ke) = \rp^2(\ke)$ as
in the definition of $\princ$
and $\rho : W \to [0,1]$ is a cut-off function supported on a collar
neighbourhood of $M$ then $\wh \rp^2(\ke) - d(\rho \prp(\ke)) \in \drpl^*(W)$
represents a pre-image of $u \in H^*(W,M)$ of $\hat \ke$, and by definition
$\lambda_W(\hat \ke \hat \ke') = u\,\hat w$.
In the notation of \eqref{eq:form_products}, we can write this as 
\[ -A_W \; = \; \int_W (\wh \rp^2 - d(\rho \prp)) \symp \wh \rp.\]
Hence,
as maps $(\gsym^2 \kc)^{m+1} \to \Q$,
\[ \int_M \prp \symp \rp^2 \; = \;
\int_W d \big( \rho \prp \symp \wh \rp^2 \big)
\; = \; A_W
\; - \; \int_W \wh \rp^2 \symp \wh \rp^2 . \]
The last term factors through
$(\gsym^2 \kc)^{m+1} \to \gsym^4 H^*(M)$, so vanishes on $\bmspam{M}$,
while the restriction of the left hand side to $\bmspam{M}$
equals $\princ_M$ by definition.
\end{proof}

By \eqref{eq:recover} this also lets us compute Massey triple products using
the coboundary $W$, so Lemma~\ref{lem:coboundary} is a generalisation of
\cite[Proposition 3.2.6]{hepworth05}.

If the cup product $H^*(M) \times H^*(M) \to H^*(M)$ is trivial
in positive degrees (so $E^* = \gsym^2 H^{>0}(M)$), then the cup-square
$\gsym^2 H^{>0}(W) \to H^*_0(W)$ and the intersection form (together with $r$)
define a degree $m{+}1$ element of $\grad^2 \grad^2 H^{>0}(M)^\vee$.
In view of Remark \ref{rmk:defect}, Lemma \ref{lem:coboundary} means that the
\bmp{} measures the failure of that 4-tensor to be fully graded symmetric.
More generally, when $M$ bounds over $H^{\leq s}$
we could use the lemma 
as the definition of the \bmp, and deduce that it is independent of the choice
of coboundary from the full graded symmetry of the quadruple cup
product on $H^*$ of closed oriented manifolds.

The coboundary perspective is useful for understanding the relation between
the \bmp{} and cohomology with integer coefficients.
Recall that $\icsq : \gsym^2 H^*(M;\Z) \to H^*(M;\Z)$
is the integral cup-square map. Let $\ikc^*$ denote the kernel of
$\icsq$ modulo torsion (or equivalently, the pre-image of $\kc^*$
under the map $\gsym^2 H^*(M;\Z) \to \gsym^2 H^*(M)$), and
$\bmspa{M;\Z} = \kersym{\ikc^*}$, the kernel of
$\gsym^2 \ikc^* \to \gsym^4 H^n(M;\Z)$.
While it is hard to see how to define an integral version of the \bmp{}
in terms of singular cochains, we may obviously define the ``integral
restriction'' $\iprinc_M : \bmspam{M;\Z} \to \Q$ as the composition of
$\princ_M$ with $\bmspam{M;\Z} \to \bmspam{M}$, as we did in the introduction.

If $M = \del W$ and $H^{\leq s}(W; \Z) \to H^{\leq s}(M; \Z)$ is onto,
we shall say that $W$ is a coboundary over $H^{\leq s}(M; \Z)$.
If $M$ is \tkc{} of dimension $m$ and $s = m{-}3n{+}1$,
then we can relate the \bmp{} of $M$ to the torsion linking form of $M$,
%
\[ b_M \colon (\gsym^2TH^*(M;\Z))^{m+1} \to \Q/\Z, \]
by a straight-forward adaptation of Hepworth's argument relating $b_M$
to Massey triple products \cite[Proposition 3.1.6]{hepworth05}.
To describe this relationship we let
$B_M : (\gsym^2 \ikc^*)^{m+1} \to \Q/\Z$ be the homomorphism induced by
\begin{equation} \label{eq:B_def}
 \ikc^i \times \ikc^{m+1-i} \to \Q/\Z, \;
(e, e') \mapsto b_M(\icsq(e), \icsq(e')) . 
\end{equation}

\begin{cor} \label{cor:princ_and_b}
Let $W$ be a closed $(m{+}1)$-manifold with \tkc{} boundary $M$, 
which is a coboundary over $H^{\leq s}(M; \Z)$ (for $s = m+1-3n$).
Pick a right inverse $\bar r : H^{\leq s}(M;\Z) \to H^{\leq s}(W;\Z)$,
and define $\bar A_W : (\gsym^2 \ikc)^{m+1} \to \Q$
analogously to $A_W$ in \eqref{eq:aw}. Then
\[ \bar A_W = B_M \mod \Z\]
and
\[ \iprinc_M = B_M|_{\bmspam{M; \Z}} \mod \Z .\]
\end{cor}

\begin{proof}
For the first equality 
we recall that if $\bar x \in H^i(W;\Z)$ and $\bar y \in H^{m+1-i}(W;\Z)$
restrict to torsion classes $x, y \in H^*(M;\Z)$, then by
\cite[Theorem 2.1]{alexander76}
\begin{equation} \label{eq:b_and_lambda}
b_M(x,y) = -\lambda_W(\bar x, \bar y).
\end{equation}
The second equality follows immediately from the first together with
Lemma \ref{lem:coboundary}. 
\end{proof}

\begin{rmk}
\label{rmk:link}
If an \tkc{} $M$ admits any coboundary over $H^{\leq s}(M; \Z)$ 
then the mod $\Z$ reduction of $\iprinc_M : \bmspam{M;\Z} \to \Q$ is
determined by the torsion linking form---in particular, if $H^*(M;\Z)$ is
torsion-free then $\iprinc$ takes integer values.
Because of the non-commutativity of the cup product on singular cochains,
we do not see a reason for this claim to be true in the absence of such
a coboundary.
\end{rmk}

\subsection{Integral realisation}

We now turn to the problem of realising a prescribed integral restriction
of the \bmp. This problem seems quite complicated in general, but by
restricting attention to the critical case of \tkc{} \fkm s we can realise a
large class of \bmp s by boundaries of $4n$-manifolds.

We focus on the following basic invariants of an \tkc{} \fkm{} $M$:
\begin{enumerate}
\item $H^{n*}(M;\Z) = H^0(M;\Z) \oplus H^n(M;\Z) \oplus H^{2n}(M;\Z)$,
the part of the cohomology ring supported in degree divisible by $n$
(capturing in particular the product 
$\icsq \colon \gsym^2H^{n}(M; \Z) \to H^{2n}(M; \Z)$);
\item the linking form, $b_M \colon TH^{2n}(M; \Z) \times TH^{2n}(M; \Z) \to \Q/\Z$;
\item the \bmp, $\iprinc_M \colon \bmspaf{M; \Z} \to \Q$.
\end{enumerate}
Note that the domain 
of $\iprinc_M$ depends only on $H^{n*}(M;\Z)$:
it equals $\bmspaf{H^{n*}(M;\Z)} := \kersym{\pol^2 \ikc^{2n}}$, where
$\ikc^{2n} := \ker(\rho \circ \icsq)$ for
$\rho \colon H^{2n}(M;\Z) \to H^{2n}(M;\Z)/T$ the projection.

Hence we define a {\em linking model} as an algebraic model of $M$,
which is a triple 
\[  (H^{n*}, b, \iprinc) \]
where $H^{n*} = H^0 \oplus H^n \oplus H^{2n}$ is a graded ring with $H^n$
torsion-free,
$b \colon T \times T \to \Q/\Z$ is a nonsingular symmetric torsion form
on the torsion subgroup $T \subseteq H^{2n}$,
and $\iprinc \colon \bmspaf{H^{n*}} \to \Q$ is a homomorphism.
Recalling Corollary \ref{cor:princ_and_b}, we say that $\iprinc$ and $b$ are {\em compatible} if
\begin{equation} \label{eq:b_and_iprinc}
\iprinc = B|_{\bmspaf{H^{n*}}} \mod \Z,
\end{equation}
where $B \colon P^2\ikc^{2n} \to \Q/\Z$ is defined from $b$ as
in \eqref{eq:B_def}.
Our main realisation result is the following

\begin{thm}
\label{thm:realiseMM}
Let $n \geq 2$, and let $(H^{n*}, b, \iprinc)$ be a linking model with
compatible $b$ and $\princ$.
Then there exists some \tkc{} closed 
$M^{4n-1}$ with an isomorphism $H^{n*}(M; \Z) \cong H^{n*}$ that identifies
$(b_M, \iprinc_M) = (b, \iprinc)$.
In addition, we may assume that $H^*(M)$ is concentrated
in degrees $\ast = 0, n, 2n{-}1, 2n, 3n{-}1$ and $4n{-}1$ and the same holds
for $H^*(M; \Z)$ when $n = 2, 4$ or $n$ is odd.
\end{thm}

Before proving Theorem \ref{thm:realiseMM} we show how it implies
Theorem \ref{thm:realiseM} of the introduction, where we used simpler algebraic
models for $M$.
A {\em torsion free model} is a pair 
\[  (F^{n*}, \iprinc), \]
where $F^{n*} = F^0 \oplus F^n \oplus F^{2n}$ is a torsion-free ring and
$\iprinc \colon \bmspaf{F^{n*}} \to \Q$ is a homomorphism.
With this terminology, Theorem \ref{thm:realiseM} states that any torsion free
model can be realised by an \tkc{} \fkm.
Given a linking model $(H^{n*}, b, \iprinc)$, we obtain a torsion free model
$(H^{n*}/T, \iprinc)$, so Theorem \ref{thm:realiseM}
follows immediately from Theorem \ref{thm:realiseMM} and the following

\begin{lem} \label{lem:b_for_princ}
Given any torsion free model $(F^{n*}, \iprinc)$ there is a linking model 
$(H^{n*}, b, \iprinc)$ with
compatible $b$ and $\iprinc$ such that $F^{n*} = H^{n*}/T$.
\end{lem}

\begin{proof}
Choose any extension of $\iprinc \colon \bmspaf{F^{n*}} \to \Q$
to $\iprinc' \colon \pol^2 \ikc^{2n} \to \Q$
and let $B_0 \colon \pol^2 \ikc^{2n} \to \Q/\Z$
be the composition of $\iprinc'$ with the canonical surjection $\Q \to \Q/\Z$.
We define $R \subseteq \ikc^{2n}$ to be the radical of~$B_0$:
\[ R : = \{ r \in \ikc^{2n} \, | \, B_0(re) = 0~\forall e \in \ikc^{2n} \}. \]
We let $S \subseteq \ikc^{2n}/R$ be the torsion subgroup and fix a 
projection $\ikc^{2n}/R \to S$, which in turn defines a surjection
$\pi \colon \ikc^{2n} \to S$.  The map $\pi$ induces a surjection 
$\pol^2(\pi) \colon \pol^2 \ikc^{2n} \to \pol^2 S$ such that 
$B_0$ factors over $\pol^2(\pi)$; \ie $B_0$ induces a (possibly singular)
symmetric torsion form
\[ b_0 \colon \pol^2 S \to \Q/\Z . \]
Set $\wh S := {\rm Hom}(S, \Q/\Z)$ and define a nonsingular torsion form $b$ by
\[ b \colon (S \times \wh S) \times (S \times \wh S) \to \Q/\Z,
\quad
b\bigl((s_1, \sigma_1),(s_2, \sigma_2) \bigr) = b_0(s_1, s_2) + 
\sigma_1(s_2) + \sigma_2(s_1). \]
To finish the proof, let $p \colon \gsym^2 F^n \to \ikc^{2n}$ be any
projection, set $(T, b) := (S \oplus \wh S, b)$,
and let $H^n = F^n$ and $H^{2n} = F^{2n} \oplus T$, with product 
\[ \gsym^2 F^n \to H^{2n} \oplus T, 
\quad ff' \mapsto (\icsq(ff'), \pi \circ p (ff')).\]
It is clear that $F^{n*} = H^{n*}/T$ as rings and that $\iprinc$
and $b$ are compatible; \ie satisfy \eqref{eq:b_and_iprinc}.
\end{proof}

To prove Theorem \ref{thm:realiseMM},
we note that it follows directly from the next two lemmas.  
Lemma \ref{lem:realiseW} is a generalisation of results
of Schmitt \cite{schmitt02} and we defer its proof to the next subsection.

\begin{lem} \label{lem:realiseW}
Let $n \geq 2$, and let a connected torsion-free ring
$G^{n*} = G^0 \oplus G^n \oplus G^{2n}$
and an even symmetric bilinear form
$\lambda \colon G^{2n} \times G^{2n} \to \Z$ be given.
Then there is a compact \tkc{} 
$4n$-manifold $W = W(G^{n*}, \lambda)$ with 
\tkc{} boundary such that the following hold:
\begin{enumerate}
\item $W$ has the homotopy type of an \tkc{} $4n$-dimensional CW complex;
\item $H^*(W)$ is concentrated in dimensions $0, n$ and $2n$;
\item $H^*(W; \Z)$ is concentrated in dimensions $0, n$ and $2n$ when $n = 2, 4$ or $n$ is odd;
\item $H^{n*}(W; \Z) = G^{n*}$;
\item $\lambda_W = \lambda$.
\end{enumerate}
\end{lem}

\begin{lem} \label{lem:realiseF}
Given a linking model $(F^{n*}, b, \iprinc)$,
we can find $(G^{n*}, \lambda)$ with $G^n = F^n$
such that the manifold
$W(G^{n*}, \lambda)$ of Lemma \ref{lem:realiseW}
has boundary $M$ with
\[ (H^{n*}(M; \Z), b_M, \iprinc_M) \cong (F^{n*}, b, \iprinc). \]
\end{lem}

\begin{proof}
By \cite[Theorem 6]{wall63} 
there is a nondegenerate symmetric bilinear form on a free abelian
group $G_2$ such that 
$\lambda_2 \colon G_2 \times G_2$ presents $-b$.  That is, there is a surjection
$\pi_2 \colon G_2 \to T$ such that 
\[ \lambda_2(y_1, y_2) = -b\bigl( \pi_2(y_1), \pi_2(y_2) \bigr) \mod \Z, \]
for all $(y_1, y_2) \in G_2 \times G_2$.
We now turn to $\icsq \colon \gsym^2 F^n \to F^{2n}$ and recall
that $\ikc^{2n} = \ker(\icsq \circ \rho) \subseteq \gsym^2 F^n$ is a summand.
We fix a projection $p \colon \gsym^2 F^n \to \ikc^{2n}$ and note that
$\icsq|_{\ikc^{2n}} \colon \ikc^{2n} \to T \subseteq F^{2n}$.
Since $\ikc^{2n}$ is free, we can choose a homomorphism
$q \colon \ikc^{2n} \to G_2$ such that
$\pi_2 \circ q = \rho \circ \icsq|_{\ikc^{2n}} \colon \ikc^{2n} \to T$.

To apply Lemma \ref{lem:realiseW}, we choose $(G^{n*}, \lambda_1)$ as follows:
\begin{enumerate}
\item Set $G^n = F^n$ and
$G^{2n} = F^{2n}/T \oplus G_2\oplus \ikc^{2n} \oplus (\ikc^{2n})^\vee$,
\item Define the product $\gsym^2 G^n \to G^{2n}$ by
$(\rho \circ \icsq, q \circ p, p, 0)$,
\item Define the symmetric bilinear form $\lambda$ by
$(G^{2n}, \lambda) = (G/T, 0) \oplus (G_2, \lambda_2) \oplus (\ikc^{2n} \oplus
(\ikc^{2n})^\vee, \lambda_3)$,
\end{enumerate}
where $\lambda_3\bigl((e_1, \alpha_1), (e_2, \alpha_2)\bigr) = 
\lambda_\ikc(e_1, e_2) + \alpha_1(e_2) + \alpha_2(e_1)$,
for $(\lambda_\ikc)$ an even symmetric bilinear form on $\ikc^{2n}$
which we shall vary as needed.
From the exact sequence
\[ \dots \to H^{4n}(W, M; \Z) \to H^{4n}(W; \Z) \to H^{4n}(M; \Z) 
\to H^{4n+1}(W; \Z) \to \dots \]
and the fact that $H^{4n+1}(W; \Z) = 0$,
we have $H^{2n}(M; \Z) = F^{2n}/T \oplus T = F^{2n}$.
By \eqref{eq:b_and_lambda} $b_M = b$.
By Lemma \ref{lem:realiseW} $H^{n*}(W; \Z) = G^{n*}$,
and so $\ker(\gsym^2 H^n(W;\Z) \to H^{2n}(W; \Z)) = \ikc$ since it is the
intersection of the kernels of $\icsq$, $p$ and $q \circ p$.
It follows that $\icsq_M \colon \gsym^2 H^n(M; \Z) \to H^{2n}(M; \Z)$ 
is identified with $\icsq$.

It remains to determine $\iprinc_M$ and we do this using Lemma \ref{lem:coboundary}.
Note that the map $A$ that computes $\princ_M$ in Lemma \ref{lem:coboundary}
is simply the map $\pol^2 \kc \to \Q$ induced by $\lambda$.
It follows that $\iprinc_M = \iprinc_2 + \iprinc_\ikc$ where $\iprinc_2$
is induced by $\lambda_2$ and $\iprinc_\ikc$ is induced by $\lambda_\ikc$.
By construction $\iprinc_M = \iprinc_2 \mod \Z$.
Hence it remains to show that $\lambda_\ikc$ can be chosen to realise
any integer-valued homomorphism $\iprinc \colon \bmspaf{M; \Z} \to \Z$,
and we do this in the following paragraphs.

Letting $\sym_0^2$ denote even symmetric bilinear forms, we thus want to
prove that the composition of
$\sym_0^2 (\ikc^{2n})^\vee \to (\pol^2 \ikc^{2n})^\vee$ with
restriction to $\bmspaf{H^*(M;\Z)}$ maps onto $\bmspaf{H^*(M;\Z)}^\vee$.
Given that there is an isomorphism
$\sym^2 \grad^2 H^n(M;\Z)^\vee \cong (\pol^2 \gsym^2 H^n(M;\Z))^\vee$,
it suffices to prove the surjectivity mod 2.

Now, the annihilator of $\sym_0^2 \grad^2 H^n(M;\Z_2)^\vee$ in
$\pol^2 \gsym^2 H^n(M;\Z_2)$ is the $\Z_2$-vector space of squares of elements
of $\gsym^2 H^n(M;\Z_2)$. That clearly intersects trivially with
$\bmspaf{M;\Z_2}$, since expanding the square of a non-zero element
$\gsym^2 H^n(M;\Z_2)$ to an element of $\gsym^4 H^n(M;\Z_2)$ can never give 0.
\end{proof}

\newcommand{\fgp}{G_1}

\subsection{Proof of Lemma \texorpdfstring{\ref{lem:realiseW}}{4.6}}

Our first step is to identify a finite \tkc{} $2n$-dimensional CW
complex $K(G^{n*})$ such that $H^0(K) \oplus H^n(K) \oplus H^{2n}(K)$
realises the prescribed ring~$G^{n*}$.
Let $F = G^n$, a free abelian group, and let $r$ be its rank.
Let $K(F^\vee, n)$ be the indicated Eilenberg-MacLane space, and
$K(F^\vee, n)^{(2n-2)}$ a $(2n{-}2)$-skeleton of $K(F^\vee, n)$.
Attach $b_{2n-2}(K(F^\vee, n)^{(2n-2)})$ $(2n{-}1)$-cells to kill
$H^{2n-2}(K(F^\vee, n)^{(2n-2)})$, calling the resulting complex~$K_0'$.
The space $K_0'$ has the rational homotopy type of $\vee_{i=1}^r S^n$ and
by the standard inductive construction of $K(F^\vee, n)$ as in \cite[Example 4.17]{hatcher02}
and Serre's Theorem on the homotopy groups of simply-connected finite CW complexes
\cite{serre53}, we may assume that $K_0'$ is a finite CW-complex.
We then set 
\[  K_0 :=  \left\{
\begin{array}{cc}
\vee_{i=1}^r S^n & \text{$n = 2, 4$ or $n$ odd,} \\
K_0' & \text{otherwise.}
\end{array}    \right. \]

\begin{lem} \label{lem:realise_lambda}
Let $G^{n*}$ be a torsion-free graded ring concentrated in degree 0, $n$
and $2n$, with $G^{2n}$ of rank $s$. 
Then for $i = 1, \dots, s$, there are maps $\phi_i\colon S^{2n-1} \to K_0$,
such for $\phi := \sqcup_{i=1}^s \phi_i$, the CW-complex
\[ K(G^{n*}) := K_0 \cup_\phi (\cup_{i=1}^s e^{2n}) \]
has $H^{n*}(K; \Z) = G^{n*}$.
\end{lem}

\begin{proof}
We give a proof that applies for both definitions of $K_0$.
We recall the $i$th-$\Gamma$-group of a finite simply-connected CW-complex $K$,
which is the group
\[ \Gamma_{i}(K) : = {\rm Im}\bigl( \pi_{i}(K^{(i-1)}) \to \pi_{i}(K^{(i)}) \bigr),  \]
where $K^{(i-1)} \to K^{(i)}$ is the inclusion of the $(i{-}1)$-skeleton of $K$
into the $i$-skeleton of $K$.  The $\Gamma$ groups lie in Whitehead's long
exact sequence 
\[ \dots \to H_{i+1}(K; \Z) \xra{b} \Gamma_i(K) \xra{i_*} \pi_i(K) \xra{\rho} H_i(K; \Z) \to \dots, \]
where $i_*$ is the obvious inclusion, $\rho$ is the Hurewicz homomorphism and $b$ is
a certain ``boundary homomorphism'': see \cite[Ch.\,2]{baues96}.
Hence for $K = K(F^\vee, n)$ we have
\[  b \colon H_{2n}(K(F^\vee, n); \Z) \cong \Gamma_{2n-1}(K(F^\vee, n)), \]
and by \cite[Theorem 3.4.3]{decker74} there is a natural surjective homomorphism 
\[ Q \colon H_{2n}(K(F^\vee, n); \Z) \to (\gsym^2 F)^\vee, \]
which is given by taking the cup squares of elements in $F = H^n(K(F^\vee, n); \Z)$
and evaluating against $H_{2n}(K(F^\vee, n); \Z)$.

Now consider $i_{0*} \colon \pi_{2n-1}(K_0) \to \Gamma_{2n-1}(K(F^\vee, n))$.
We claim that $Q \circ b^{-1} \circ i_{0*}$ is onto.  When $K_0 = K(F^\vee, n)^{(2n-2)}$,
we have that $i_{0*}$ is onto by definition and so $Q \circ b^{-1} \circ i_{0*}$ is
onto.
When $K_0 = \vee_{i=1}^r S^n$, we let $\{x_1, \dots, x_r\}$ be a basis for $F$.
For $i \neq j$, the element $[x_ix_j] \in \gsym^2 F$ can be realised by Whitehead products
$[\iota_i, \iota_j]$ where $\iota_k \colon S^n \to \vee_{i=1}^r S^n$ is the inclusion
of the $k$th summand.  When $n = 2, 4$, elements of the form $[x_i^2] \in \gsym^2 F$
can be realised by maps $\iota_i \circ h$, where $h \colon S^{2n-1} \to S^n$
has Hopf-invariant $1$.  For more details see \cite[Prosition 3.11]{schmitt02} 
in the case $n = 2$, the case $n = 4$ is analogous.
Hence we choose $\phi_i \in \pi_{2n-1}(K_0)$ such that 
\[ Q \circ b^{-1} \circ i_{0*}(\phi_i) = \icsq^\vee(y_i^\vee), \]
where $\icsq^\vee \colon (G^{2n})^\vee \to (\gsym^2 F)^\vee$ is the dual to the
product map $\icsq \colon \gsym^2 G^n \to G^{2n}$ and
$\{y_1^\vee, \dots, y_s^\vee\}$ is a basis for $(G^{2n})^\vee$.
By construction, we may then identify $H^{2n}(K(G^{n*}); \Z) = G^{2n}$ and
the cup product structure on $H^{n*}(K(G^{n*});\Z)$ is given by the ring
structure of $G^{n*}$. 
\end{proof}

When $n = 2$, the construction of the manifolds 
$W = W(G^{n*}, \lambda)$ follows
easily from the results of Schmitt \cite[\S 3]{schmitt02} which build on handlebody theory
and classical embedding results of Haefliger.  When $K_0 = K(F^\vee, n)^{(2n-2)}$, we may
have many layers of handles to attach, and it is convenient to use the theory 
of thickenings as developed by Wall \cite{wall66}.  We briefly recall the notion of a thickening:
Let $K$ be a simply-connected finite connected CW-complex. 
An $m$-thickening of $K$ is a pair $(W, \phi)$
where $W$ is a compact $m$-manifold with simply-connected boundary $\del W$ 
and $\phi \colon K \to W$ is a homotopy equivalence.
Since the map $\phi$ will be clear in our arguments from the discussion, 
we suppress it and from the notation and call $W$ a thickening of $K$.

Now let $(G^{n*}, \lambda)$ be as in the hypotheses of
Lemma \ref{lem:realiseW} and
apply Lemma \ref{lem:realise_lambda} 
to obtain the $2n$-complex $K = K(G^{n*})$.
By \cite[\S 3 Trivial thickening]{wall66}, there are unique $4n$-thickenings $W_0(K)$ of $K$ 
and $W_0(K_0)$ of $K_0$ which are compact submanifolds of $\R^{4n}$ and which are called
{\em trivial thickenings}.
Moreover, by \cite[Suspension Theorem]{wall66}, $W_0(K) \cong W'_0(K) \times D^1$,
where $W_0'(K) \subset \R^{4n-1}$ also thickens $K$.  
It follows that $\lambda_{W_0(K)} = 0$ is the trivial form.
By assumption, the required form $\lambda$ on 
$W(G^{n*}, \lambda)$ is an even form,
so it suffices to show how to modify the intersection form of $W_0(K)$ by any even form,
without changing the cup-product structure.

By construction, $W_0(K) = W_0(K_0) \cup_\phi(\cup_{i=1}^s h^{2n}_i)$ is obtained from the trivial thickening
of $K_0$ by attaching $s$ $2n$-handles $h^{2n}_i \cong D^{2n} \times D^{2n}$ along a framed embedding
\[ \phi 
\colon \coprod_{i=1}^s (D^{2n} \times S^{2n-1}) \hookrightarrow \del W_0(K_0),  \]
where we attach one handle for each element of a basis of $G^{2n}$, 
which we assume has rank~$s$.  
Now by \cite[Lemma 1]{wall62}, every even symmetric bilinear form $l$ is realised
as the intersection form of handlebody $W_l = D^{4n} \cup_{\phi_0}(\cup_{i=1}^s h^{2n}_{0j})$
which is obtained by attaching $2n$-handles $h^{2n}_{0i}$ along a framed embedding
\[ \phi_0 
\colon \coprod_{i=1}^s  (D^{2n} \times S^{2n-1}) \hookrightarrow D^{4n-1} \subset S^{4n-1}. \]
We take $W_l$ to have the intersection form $(G^{2n}, \lambda)$.
Fixing an embedding $D^{4n-1} \hookrightarrow \del W_0(K_0)$
disjoint from ${\rm Im}(\phi)$, we then form the framed embedding
$\phi' = \phi + \phi_0$ by tubing together the components of $\phi$ and $\phi_0$.
We define
\[ W(G^{n*}, \lambda) : = W_0(K_0) \cup_{\phi'}(\cup_{i=1}^s h^{2n}_i) \]
to be the manifold obtained by attaching $2n$-handles to $W_0(K_0)$ along $\phi'$.
Since ${\rm Im}(\phi_0) \subset D^{4n-1}$, $W = W(G^{n*}, \lambda)$
has that same homotopy type as $W_0(K)$ and hence the same cup-product structure.
On the other hand, the intersection form of $W$ is identified with the intersection
form of $W_l$ which is the intersection form required for Lemma \ref{lem:realiseW}. 
This completes the proof of Lemma \ref{lem:realiseW}
(and hence of Theorem \ref{thm:realiseMM}).

\begin{rmk}
\label{rmk:5nW}
If we put aside the \bmp{} and focus on realising cohomology algebras with
certain features, then the ideas in the proof of Lemma \ref{lem:realiseW}
extend to higher dimensions.
For instance, suppose that $n \geq 2$ and $G^0 \oplus G^n \oplus G^{2n}$ is a
torsion-free graded ring, and that we wish to realise $G^{n*} \otimes \Q$
as $H^{\leq 2n}(M)$ of a closed \tkc{} $m$-manifold $M$ with $m \geq 4n{+}1$.
Then we can take $K$ as in Lemma \ref{lem:realise_lambda}, let
$W^m_0(K)$ be the trivial $m$-dimensional thickening of $K$,
and let $M : = \del(W^m_0(K) \times I) = W_0^m(K) \cup_{\Id} W_0^m(K)$.
Then $H^{*}(K; \Z) \cong H^{*}(W;\Z) \cong H^{*}(M; \Z)$ for
$\ast \leq m{-}2n{-}1$, so $H^{\leq 2n}(M) = G^{n*} \otimes \Q$.
\end{rmk}

\section{Applications to \texorpdfstring{\fkm s}{(4n-1)-manifolds}}
\label{sec:applications}

In this section we discuss applications of the \bmp.
We begin with the proof of
Theorem \ref{thm:lef} from the introduction, and then give examples
where the \bmp{} is non-trivial despite all Massey products vanishing
(or there not being any defined triple products at all).

We restrict our attention to \tkc{} \fkm s. Then the space $\bmspaf{M}$,
on which the significant components of the \bmp{} are defined,
involves only $H^n(M)$ and the kernel $\kc^{2n}$ of the cup product
$\calg^2 H^n(M) \to H^{2n}(M)$ (to make sense of the graded power~$\calg^2$, we
interpret $H^n(M)$ as a graded vector space concentrated in degree $n$).
The applications essentially reduce to understanding the details of
how $\bmspaf{M}$ depends on $\kc^{2n}$.

In the final section we 
briefly discuss the role of \bmp{} in the classification of
simply-connected spin 7-manifolds.

\subsection{Intrinsic formality and the hard Lefschetz property}

We now prove Theorem \ref{thm:lef}, on the intrinsic formality of closed
\tkc{} \fkm s with $b_3 \leq 3$ and a hard Lefschetz property.
In view of Corollary \ref{cor:intrinsic} it suffices to prove that
$\bmspaf{M} = 0$.
By Poincar\'e duality, the hard Lefschetz property is equivalent to
equivalent to $\ann \kc^{2n} \subseteq \gsym^2 H^n(M)^\vee$ containing a
non-degenerate bilinear form $q$. Hence Theorem \ref{thm:lef}
is a consequence of the following algebraic result.

\begin{prop}
\label{prop:lef}
Let $V$ be a graded vector space of dimension $\leq 3$, concentrated in
degree $n$. Let $\kc$ be a subspace of $\gsym^2 V$.
If $\ann \kc \subseteq \grad^2 V^\vee$ contains a non-degenerate element $q$,
then $\kersym{\pol^2\kc} = 0$.
\end{prop}

\begin{proof}
It is convenient to consider the dual picture. By the duality of the sequences
\eqref{eq:dualseq}, $\kersym{\gsym^2\kc} = 0$ if and only if the restriction of
$\ksmap : \sym^2 \grad^2 V^\vee \to \kerasym{\sym^2\agrad^2V^\vee}$ to
$\ann \pol^2 \kc \subseteq \sym^2 \grad^2 V^\vee$ is surjective.
In terms of the Kulkarni-Nomizu product $\owedge$ described in
Remark \ref{rmk:kn},
the image $\ksmap(\ann \pol^2 \kc)$ is $\ann \kc \owedge \grad^2 V^\vee$.

The case when $n$ is odd is essentially trivial, because then $q$ is a
symplectic form on $V$ and so $\dim V = 0$ or 2.
If $\dim V = 2$ then $\kerasym{\sym^2\sym^2V^\vee}$ is one-dimensional, and it
is easy to
see that $q \owedge q$ is non-zero. In particular
$q \owedge \alt^2 V^\vee = \kerasym{\sym^2\sym^2V^\vee}$, so
$\kersym{\pol^2\kc} = 0$.

For the case when $n$ is even, Besse \cite[1.119]{besse87} explains that
$q \owedge \sym^2 V^\vee$ is all of $\kerasym{\sym^2\alt^2V^\vee}$ for
$\dim V \leq 3$; this is the same algebraic result that leads to the well-known
fact from Riemannian geometry that the Riemann curvature is determined by the
Ricci curvature in dimension $\leq 3$.
\end{proof}

\subsection{\bmp s without Massey products}
\label{subsec:nonordinary}

Let us now consider the question of when the \bmp{} can be non-trivial even
though all Massey triple products vanish. According to Lemma \ref{lem:recover},
Massey triple products on a closed oriented $m$-manifold $M$
correspond to evaluating the \bmp{} on elements of $\bmspam{M}$
that are ordinary in the sense of Definition \ref{def:ordinary}.
Therefore the algebraic version of the question is whether there exist
$m$-dimensional Poincar\'e duality algebras $H^*$ where 
$\bmspam{H^*}$ is not generated by ordinary elements.

As above, we restrict to the case when $M$ is \tkc{} of dimension $4n{-}1$.
Then $\bmspaf{M} = \kersym{\pol^2 \kc^{2n}}$, so the question boils down to
whether for a vector space $V$ there is a subspace $\kc$ of $\pol^2 V$ or
$\ext^2 V$ (according to whether $n$ is even or odd) such that
$\kersym{\pol^2 \kc}$ is not generated by ordinary elements.

Let us begin with the case when $n$ is even. Example \ref{ex:rk5} gives an
example where $r := \dim V = 5$, and $\kersym{\pol^2 \kc}$ is non-trivial even
though it contains \emph{no} ordinary elements at all (giving rise to
\tkc{} \fkm s where any Massey triple products that are defined are forced
to vanish due to the symmetries---\eg ones of the form $\gen{x,y,x}$---but
can still have a non-trivial \bmp).
We find it helpful to first present a dimension-counting argument to show that
this is not an uncommon phenomenon when $r > 5$.

\begin{lem}
\label{lem:strict}
For any $r := \dim V \geq 6$ there exist $\kc \subseteq \pol^2 V$ such that
$\kersym{\pol^2 \kc}$ is non-trivial but contains no ordinary elements.
\end{lem}

\begin{proof}
Note that the condition that $\kersym{\pol^2 \kc}$ contain an ordinary element
means that there are some 2-planes $A, B \subseteq V$ such that
$A \poltp B := \{ x \poltp y \mid x \in A, y \in B \} \subseteq \pol^2 V$ is
contained in~$\kc$. Let us first consider the case when
$A \not= B$, so that $A \poltp B$ has dimension 4 rather than 3.
$Gr_2(V) \times Gr_2(V)$ has dimension
$4r{-} 8$, and the space of $k$-planes $\kc \subseteq \pol^2 V$ containing a
fixed $A \poltp B$ has dimension $(k{-}4)\left({r+1 \choose 2} - k \right)$.
So the space of $k$-planes containing some $A \poltp B$ with $A \not= B$
has positive codimension in $Gr_k(\pol^2 V)$ if
\[ (k-4)\left({r+1 \choose 2} - k \right) + 4r - 8
\; < \; k\left({r+1 \choose 2} - k\right) \]
which reduces to
\[ r-2 \; < \; {r+1 \choose 2} - k . \]
Similarly, for the the space of $k$-planes containing some $\pol^2 A$ to have
positive codimension in $Gr_k(\pol^2 V)$ reduces to
\[ 2r-4 \; < \; 3 \left({r+1 \choose 2} - k\right) , \]
which is weaker than the above condition. Thus if $k \leq {r \choose 2} + 1$
and $\kc \in Gr_k(\pol^2 V)$ is generic then $\kersym{\pol^2 \kc}$ contains no
ordinary elements. 
Now for $k := {r \choose 2} + 1$ and $r \geq 6$
\[
\dim \pol^2 \kc \; = \;
{{{r \choose 2} + 2} \choose 2} \; \geq \;
{{k+3} \choose 4} \; = \; \dim \pol^4 V.
\]
Hence in this case any $\kc \in Gr_k(\pol^2 V)$ has $\kersym{\pol^2 \kc}$
non-trivial.
\end{proof}

When $k = {r \choose 2} + 2$, \ie when $\kc$ has codimension $r-2$, the
expected codimension of the space of $k$-planes $\kc \subset \pol^2 V$
containing a fixed $A \poltp B$ equals the dimension of
$Gr_2(V) \times Gr_2(V)$, and we would expect each $\kc$ to contain a finite
number of $A \poltp B$. We can turn this round as follows.

For each $(A, B) \in Gr_2(V) \times Gr_2(V)$ and
$q \in \ann \kc \subset \pol^2 V^*$, the condition that
$A \poltp B \subseteq \kc$
implies that $A$ and $B$ are orthogonal with respect to the bilinear form $q$,
which imposes 4 constraints on $(A,B)$. Now recall that mapping an
oriented 2-plane with orthonormal basis $x, y$ (w.r.t. some standard inner
product) to $\gen{x+iy} \in \PP(V \otimes \C)$ embeds
$\ogr_2(V) \into \CP^{r-1}$ with
image the quadric $Q := \{ z : \sum z_i^2 = 0 \}$. That $A$ and $B$ are
$q$-orthogonal translates to the condition that the image of $(A,B)$ in
$\CP^{r-1} \times \CP^{r-1}$ lies in the subset of $Q \times Q$ cut out by
$q(z,z) = q(z, \bar z) = 0$. These correspond to sections of the line bundles
$\calo(1, 1)$ and $\calo(1, -1)$ respectively (but only the first is
holomorphic).

Writing $c_1$ and $c_2$ for the generators of the $H^2$ of the
two $\CP^{r-1}$ factors, we see that the topological intersection number of $Q
\times Q$ with $r-2$ such subsets is the $(c_1c_2)^{r-1}$ coefficient of
${(2c_1)(2c_2)}(c_1 {+} c_2)^{r-2}(c_1 {-} c_2)^{r-2}$, which equals
$\pm 4{{r - 2} \choose {\frac{r}{2} - 1}}$ when $r$ is even, and vanishes when
$r$ is odd. This counts each ordered pair of unoriented 2-planes $(A,B)$ with
$A \poltp B \subseteq \kc$ 4 times, and possibly with some cancelling signs.

\begin{itemize}
\item For $r = 3$ we expect that a generic codimension 1 subspace
$\kc \subset \pol^2 V$ should contain no $A \poltp B$. Indeed, if the generator
$q \in \ann \kc$ is non-degenerate then there can obviously be no
$q$-orthogonal 2-planes (and moreover we already argued in Proposition 
\ref{prop:lef} that $\kersym{\pol^2 \kc} = 0$ in that case). 

\item For $r = 4$ we expect that for any $\kc \subset \pol^2 V$ of
codimension 2, there should be at least two ordered pairs $(A,B)$ of
unoriented 2-planes such that $A \poltp B$ is contained in $\kc$
(and possibly only one unordered pair).
If $\ann \kc$ is spanned by two non-degenerate elements, then we could also see
this by applying the Lefschetz fixed point theorem to the composition of the
maps $\perp : Gr_2(V) \to Gr_2(V)$ that they define
(since $\chi(Gr_2(V)) = 2$).

If $\ann \kc$ is spanned by $x_1^2 +x_2^2 + x_3^2 + x_4^2$ and
$\lambda_1x_1^2 + \lambda_2x_2^2 + \lambda_3x_3^2 + \lambda_4x_4^2$
with $\lambda_i$ distinct, then the coordinate planes corresponding to each
partition of $\{1,2,3,4\}$ into two halves gives 6 ordered pairs of
simultaneously orthogonal planes.

\item For $r = 5$ the calculation suggests that there may be some choices
of $q_1, q_2, q_3 \in \sym^2 V^*$ such that
$\kc := \ann \gen{q_1, q_2, q_3} \subseteq \pol^2 V$
does not contain any $A \poltp B$.
\end{itemize}

\begin{ex}
\label{ex:rk5}
Let $\kc := \ann \gen{q_1, q_2, q_3} \subseteq \pol^2 \Q^5$
for $q_i \in \sym^2 \Q^5$ defined by
\begin{align*}
q_1 &= x_1x_4 + x_3x_5 , \\
q_2 &= x_2x_5 + x_3x_4 , \\
q_3 &= x_1^2 + x_1x_2 + x_2^2 + x_3^2 + x_4^2 + x_5^2 .
\end{align*}
Suppose that $A, B$ are orthogonal with respect to $q_1$ and $q_2$.
Let $\pi \subset \Q^5$ be the 3-dimensional subspace $\{ x_4 = x_5 = 0 \}$.
One can check that
\begin{enumerate}
\item If $A$ is contained in $\pi$ then so is $B$,
and vice versa.
\item In fact, if $A$ intersects $\pi$ non-trivially then $B$ is contained
in $\pi$, except if $A\cap \pi$ is spanned by an element
of the form $(a^2, b^2, \pm ab, 0, 0)$.
\item If $A$ and $B$ are both transverse to $\pi$ then they
are equal.
\end{enumerate}
Now if $A$ and $B$ are both contained in $\pi$ then they are definitely not
orthogonal with respect to $q_3$, because its restriction to $\pi$ is
non-degenerate. If $A$ and $B$ are equal they also cannot be $q_3$-orthogonal
because $q_3$ is positive-definite. Finally,
\begin{align*}
q_3\big( (a^2, b^2, \pm ab, 0, 0), (c^2, d^2, \pm cd, 0, 0) \big) 
\; = &\;  a^2c^2 + \half(a^2d^2 + b^2c^2) + b^2d^2 \pm abcd \\
= &\;  \half(a^2 + b^2)(c^2 + d^2) + \half (ac \pm bd)^2 > 0
\end{align*}
for any non-zero $(a,b)$ and $(c,d)$. Hence $\kc$ contains no $A \poltp B$, but
$\kersym{\pol^2 \kc}$ has dimension at least
${13 \choose 2} - {8 \choose 4} = 78 - 70 = 8$.
\end{ex}
We can also show that it happens for $r \geq 11$ that $\kersym{\pol^2 \kc}$ is
non-trivial even though $\kc$ does not even contain any monomials $x \poltp y$.
The image of $\PP(V) \times \PP(V) \to \PP(\pol^2 V),
\; (\gen{x}, \gen{y}) \mapsto \gen{x \poltp y}$
has dimension $2r-2$, so is disjoint from a generic $\kc \subseteq \pol^2 V$ of
dimension ${{r+1} \choose 2} - 2r + 1$.
For $r \geq 11$, this makes $\dim \pol^2 \kc > \dim \pol^4 V$, so $\kc$
automatically has $\kersym{\pol^2 \kc}$ non-trivial.
This gives rise to $(2k{-}1)$-connected $(8k{-}1)$-manifolds which can have a
non-trivial \bmp, even though there are no defined Massey triple products at
all.

\pagebreak[2]
Now let us consider the case when $n$ is odd. In this case, the dimension count
argument is not especially sharp. For $r = 4$, the fact that $Gr_2(V)$ is
4-dimensional leads one to ``expect'' a generic 2-dimensional
$\kc \subseteq \ext^2 V$ to contain (the one-dimensional) $\ext^2 A$ for
some $A \in Gr_2(V)$.  Nevertheless we have the following simple example.

\begin{ex}
\label{ex:rk4}
Let $V = \Q^4$ and
\[ \kc := \gen{v_1 \wedge v_2 + v_3 \wedge v_4, \,
v_1 \wedge v_3 - v_2 \wedge v_4, \, v_1 \wedge v_4 + v_2 \wedge v_3}
\subset \ext^2 V , \]
for a basis $v_1, \ldots, v_4 \in V$.
Then $\kersym{\pol^2 \kc}$ does not contain any ordinary elements: indeed
$\kc$ contains no decomposable elements at all, so for any $x, y \in V$,
$x \wedge y \in \kc$ implies that $x$ and $y$ are linearly dependent.
However, $\dim \pol^2 \kc = 6$ while $\dim \ext^4 V = 1$, so
$\dim \kersym{\pol^2 \kc} = 5$ (it is clear that $\pol^2 \kc$ maps onto
$\ext^4 V$).
\end{ex}

\begin{rmk}
We emphasise that by Theorem \ref{thm:realisation}, every triple 
$(V, \kc, \princ \in \kersym{\pol^2\kc}^\vee)$
covered by Lemma \ref{lem:strict} and Examples \ref{ex:rk5} and \ref{ex:rk4} is
realised as $\left(H^n(M), \ker(\gsym^2 H^n(M) \to H^{2n}(M)), \princ_M\right)$
for some closed \tkc{} \fkm{} $M$.  A corresponding integral statement
follows from Theorem \ref{thm:realiseM}.
\end{rmk}

\subsection{Some remarks on the classification of simply-connected spin 7-manifolds}
\label{subsec:7mfds}

We begin this subsection by determining which linking models
and Pontrjagin classes are realised by simply-connected
spin $7$-manifolds $M$.
Let $p_M \in H^4(M; \Z)$ be the spin characteristic class,
related to the first Pontrjagin class by $2p_M = p_1(M)$.
By \cite[Lemma 2.2(i)]{7class}, $p_M$ is always even.

\begin{prop}
\label{prop:1c7mFA}
Given a linking model $(H^{2*}, b, \iprinc)$ and $p \in 2H^4$,
there is a $1$-connected spin $7$-manifold $M$ with
\[ (H^{2*}(M; \Z), b_M, \iprinc_M, p_M) = 
(H^{2*}, b, \iprinc, p) \]
if and only if $b$ and $\iprinc$ are compatible.  Moreover, 
we may always assume that $TH^3(M; \Z) = 0$.
\end{prop}

\begin{proof}
Given $M$, we first show that the linking form $b_M$
and \bmp{} $\princ_M$ are compatible.
Let $K(H^2(M; \Z), 2)$ be the indicated Eilenberg-MacLane space.
By \mbox{\cite[Proposition 4.2]{crowley14a}} the spin bordism group 
$\Omega_7^{spin}\left(K(H^2(M; \Z), 2)\right)$ vanishes.
Hence $M$ bounds over $H^2(M;\Z)$ and by Corollary \ref{cor:princ_and_b}
$b_M$ and $\iprinc_M$ are compatible.

The existence statement for any linking model $(H^{2*}, b, \iprinc)$
with $b$ and $\iprinc$ compatible
follows immediately from Theorem \ref{thm:realiseMM}.
Hence it remains to determine the possible values of $p_M$.
By \mbox{\cite[\S 3]{schmitt02}} the manifolds $M$ of
Theorem \ref{thm:realiseMM} realising a given $(H^{2*}, \iprinc)$ can be
assumed spin with $p_M$ any element of $2H^4(M; \Z) \cong 2H^4$.
\end{proof}

Corollary \ref{cor:1c7m_up_to_finite} implies in particular that the
invariants in Proposition \ref{prop:1c7mFA} determine the
diffeomorphism type of $M$ up to a finite number of possibilities.
We conclude with a discussion of how to pin down the
remaining finite ambiguity.

The further invariants needed include the quadratic linking family
and generalised Eells-Kuiper invariant from the 2-connected classification
\cite{7class}.
When $M$ bounds over its normal 2-type in the sense of Kreck \cite{kreck99}
(in particular, whenever $\pi_2(M)$ is torsion-free), one can adapt the
coboundary description of the Bianchi-Massey tensor from Lemma
\ref{lem:coboundary} to define mod $m$ 
extensions of $\iprinc$ for any integer~$m$. 
One should further expect to be able to use such coboundaries to define some
further generalised version of the generalised Kreck-Stolz invariants of
Hepworth \cite{hepworth05}, which are based on \cite[Theorem~6]{kreck99}.
Assuming that $\pi_2(M)$ is torsion-free, so that
$H^2(M;\Z/m) \cong H^2(M;\Z) \otimes \Z/m$, may help evade some subtleties.
(Note also that Proposition \ref{prop:1c7mFA} lets us realise every possible
integral cohomology ring compatible with $\pi_2 M$ being
torsion-free.)

\begin{conj}[\cf {\cite[\S 2.c]{crowley13}}]
Simply-connected spin 7-manifolds $M$ with torsion-free $\pi_2(M)$ are
classified up to spin diffeomorphism by the cohomology ring $H^*(M; \Z)$, the
spin characteristic class $p_M$, the torsion linking form $b_M$ on $TH^4(M;\Z)$
and its family of quadratic refinements, the generalised Eells-Kuiper invariant
from \cite{7class}, the \bmp{} $\iprinc_M$ and its
mod $m$ extensions, and some variation of the generalised Kreck-Stolz
invariants from~\cite{hepworth05}. 
\end{conj}


\bibliographystyle{amsinitial}
\bibliography{g2geom}

\providecommand{\bysame}{\leavevmode\hbox to3em{\hrulefill}\thinspace}
\providecommand{\MR}{\relax\ifhmode\unskip\space\fi MR }
\providecommand{\MRhref}[2]{%
  \href{http://www.ams.org/mathscinet-getitem?mr=#1}{#2}
}
\providecommand{\href}[2]{#2}
\begin{thebibliography}{10}

\bibitem{alexander76}
J.~P. Alexander, G.~C. Hamrick, and J.~W. Vick, \emph{Linking forms and maps of
  odd prime order}, Trans. Amer. Math. Soc. \textbf{221} (1976), no.~1,
  169--185.

\bibitem{amann15}
M.~Amann, \emph{Rational homotopy theory}, in preparation, 2018.

\bibitem{amann12}
M.~Amann and V.~Kapovitch, \emph{On fibrations with formal elliptic fibers},
  Advances in Mathematics \textbf{231} (2012), no.~3, 2048--2068.

\bibitem{barge76}
J.~Barge, \emph{Structures diff\'{e}rentiables sur les types d'homotopie
  rationnelle simplement connexes}, Ann. Sci. \'{E}cole Norm. Sup. (4)
  \textbf{9} (1976), no.~4, 469--501.

\bibitem{baues96}
H.-J. Baues, \emph{Homotopy type and homology}, Oxford Mathematical Monographs,
  The Clarendon Press Oxford University Press, New York, 1996, Oxford Science
  Publications.

\bibitem{besse87}
A.~Besse, \emph{Einstein manifolds}, Springer--Verlag, New York, 1987.

\bibitem{crowley14a}
J.~Bowden, D.~Crowley, and A.~I. Stipsicz, \emph{The topology of {S}tein
  fillable manifolds in high dimensions {I}}, Proc. Lond. Math. Soc. (3)
  \textbf{109} (2014), no.~6, 1363--1401.

\bibitem{bmm18}
U.~Buijs, J.~M. Moreno-Fernández, and A.~Murillo, \emph{A-infinity structures
  and {M}assey products}, arXiv:1801.03408, 2018.

\bibitem{cavalcanti06}
G.~Cavalcanti, \emph{Formality of $k$-connected spaces in $4k+3$- and
  $4k+4$-dimensions}, Math. Proc. Camb. Phil. Soc. \textbf{141} (2006),
  101--112.

\bibitem{crowley13}
D.~Crowley and S.~Goette, \emph{Kreck-{S}tolz invariants for quaternionic line
  bundles}, Trans. Amer. Math. Soc. \textbf{365} (2013), no.~6, 3193--3225.

\bibitem{7class}
D.~Crowley and J.~Nordström, \emph{The classification of 2-connected
  7-manifolds}, Proc. Lond. Math. Soc. \textbf{119} (2019), 1--54.

\bibitem{decker74}
G.~J. Decker, \emph{The integral homology algebra of an {E}ilenberg-{M}ac{L}ane
  space}, Ph.D. thesis, University of Chicago, 1974.

\bibitem{deligne75}
P.~Deligne, P.~Griffiths, J.~Morgan, and D.~Sullivan, \emph{Real homotopy
  theory of {K}ähler manifolds}, Invent. Math. \textbf{29} (1975), 245--274.

\bibitem{dranishnikov05}
A.~N. Dranishnikov and Y.~B. Rudyak, \emph{Examples of non-formal closed
  {$(k-1)$}-connected manifolds of dimensions {$\ge4k-1$}}, Proc. Amer. Math.
  Soc. \textbf{133} (2005), no.~5, 1557--1561.

\bibitem{felix01}
Y.~F{\'e}lix, S.~Halperin, and J.-C. Thomas, \emph{Rational homotopy theory},
  Graduate Texts in Mathematics, vol. 205, Springer-Verlag, New York, 2001.

\bibitem{fernandez05}
M.~Fern\'andez and V.~Muñoz, \emph{Formality of {D}onaldson submanifolds},
  Math. Z. \textbf{250} (2005), 149--175.

\bibitem{halperin79}
S.~Halperin and J.~Stasheff, \emph{Obstructions to homotopy equivalences}, Adv.
  Math. \textbf{32} (1979), 233--279.

\bibitem{hatcher02}
A.~Hatcher, \emph{Algebraic topology}, Cambridge University Press, Cambridge,
  2002.

\bibitem{hepworth05}
R.~Hepworth, \emph{Generalized {K}reck-{S}tolz invariants and the topology of
  certain 3-{S}asakian 7-manifolds}, Ph.D. thesis, University of Edinburgh,
  2005.

\bibitem{kadeishvili80}
T.~Kadeishvili, \emph{On the theory of homology of fiber spaces}, Uspekhi Mat.
  Nauk \textbf{35} (1980), no.~3(213), 183--188, International Topology
  Conference (Moscow State Univ., Moscow, 1979).

\bibitem{kadeishvili09}
\bysame, \emph{Cohomology ${C}_\infty$-algebra and rational homotopy type},
  Algebraic topology---old and new, Banach Center Publ., vol.~85, Polish Acad.
  Sci. Inst. Math., Warsaw, 2009, pp.~225--240.

\bibitem{kreck99}
M.~Kreck, \emph{Surgery and duality}, Ann. of Math. (2) \textbf{149} (1999),
  no.~3, 707--754.

\bibitem{kreck91a}
M.~Kreck and G.~Triantafillou, \emph{On the classification of manifolds up to
  finite ambiguity}, Canad. J. Math. \textbf{43} (1991), no.~2, 356--370.

\bibitem{loday12}
J.-L. Loday and B.~Vallette, \emph{Algebraic operads}, Grundlehren der
  matematischen Wissenschaften, vol. 346, Springer--Verlag, Berlin, 2012.

\bibitem{miller79}
T.~J. Miller, \emph{On the formality of $k-1$-connected compact manifolds of
  dimension less than or equal to $4k-2$}, Illinois J. Math. \textbf{23}
  (1979), 253--258.

\bibitem{munoz14}
V.~Muñoz and A.~Tralle, \emph{Simply-connected {K}-contact and {S}asakian
  manifolds of dimension 7}, Math. Z. \textbf{281} (2015), 457--470.

\bibitem{schmitt02}
A.~Schmitt, \emph{On the classification of certain piecewise linear and
  differentiable manifolds in dimension eight and automorphisms of
  $\#_{i=1}^b({S}^2\times {S}^5)$}, Enseign. Math. \textbf{48} (2002), no.~3-4,
  263--289.

\bibitem{serre53}
J.-P. Serre, \emph{Groupes d'homotopie et classes de groupes ab\'eliens}, Ann.
  of Math. (2) \textbf{58} (1953), 258--294.

\bibitem{su09}
Z.~Su, \emph{Rational homotopy type of manifolds}, Ph.D. thesis, Indiana
  University, 2009.

\bibitem{sullivan77}
D.~Sullivan, \emph{Infinitesimal computations in topology}, Publ. math.
  I.H.É.S. \textbf{47} (1977), 269--331.

\bibitem{taylor10}
L.~Taylor, \emph{Controlling indeterminacy in {M}assey triple products}, Geom.
  Dedicata. \textbf{148} (2010), 371--389.

\bibitem{vallette12}
B.~Vallette, \emph{{A}lgebra+{H}omotopy={O}perad}, arXiv:1202.3245.

\bibitem{wall62}
C.~T.~C. Wall, \emph{Classification of {$(n{-}1)$}--connected
  {$2n$}--manifolds}, Ann. of Math. (2) \textbf{75} (1962), 163--189.

\bibitem{wall63}
\bysame, \emph{Quadratic forms on finite groups, and related topics}, Topology
  \textbf{2} (1963), 281--298.

\bibitem{wall66}
\bysame, \emph{Classification problems in differential topology. {IV}.
  {T}hickenings}, Topology \textbf{5} (1966), 73--94.

\bibitem{wall99}
\bysame, \emph{Surgery on compact manifolds}, second ed., Mathematical Surveys
  and Monographs, vol.~69, American Mathematical Society, Providence, RI, 1999,
  Edited and with a foreword by A. A. Ranicki.

\end{thebibliography}

\end{document}